\documentclass[reqno, 11pt, a4paper]{amsart}

\usepackage{latexsym,ifthen,xspace}
\usepackage{amsmath,amssymb,amsthm}
\usepackage{enumerate, calc}
\usepackage[colorlinks=true, pdfstartview=FitV, linkcolor=blue,
  citecolor=blue, urlcolor=blue, pagebackref=false]{hyperref}
\usepackage[text={33pc,605pt},centering]{geometry}    
\usepackage{graphicx}

\newtheorem{theorem}{Theorem}[section]
\newtheorem{lemma}[theorem]{Lemma}
\newtheorem{prop}[theorem]{Proposition}

\newtheorem{corro}[theorem]{Corollary}

\theoremstyle{definition}
\newtheorem{definition}[theorem]{Definition}

\theoremstyle{remark}
\newtheorem{remark}[theorem]{Remark}

\numberwithin{equation}{section}

\DeclareMathAlphabet{\mathsl}{OT1}{cmss}{m}{sl}
\SetMathAlphabet{\mathsl}{bold}{OT1}{cmss}{bx}{sl}

\newcommand{\om}{\ensuremath{\omega}}

\newcommand{\bbE}{\ensuremath{\mathbb E}}

\newcommand{\bbN}{\ensuremath{\mathbb N}} 
 
\newcommand{\bbP}{\ensuremath{\mathbb P}} 
 
\newcommand{\bbR}{\ensuremath{\mathbb R}}

\newcommand{\bbZ}{\ensuremath{\mathbb Z}} 
\DeclareMathOperator{\supp}{\mathrm{supp}}

\def\indicator{{\mathchoice {1\mskip-4mu\mathrm l}%
{1\mskip-4mu\mathrm l}{1\mskip-4.5mu\mathrm l}%
{1\mskip-5mu\mathrm l}}}

\begin{document}

\title[Continuity and estimates of the Liouville heat kernel]{Continuity and estimates of the Liouville heat kernel with applications to spectral dimensions}


\author{Sebastian Andres}
\address{Rheinische Friedrich-Wilhelms Universit\"at Bonn}
\curraddr{Endenicher Allee 60, 53115 Bonn}
\email{andres@iam.uni-bonn.de}
\thanks{S.A.\ was partially supported by the CRC 1060 \lq\lq The Mathematics of Emergent Effects\rq\rq, Bonn.}

\author{Naotaka Kajino}
\address{Department of Mathematics, Graduate School of Science, Kobe University}
\curraddr{Rokkodai-cho 1-1, Nada-ku, Kobe 657-8501, Japan.}
\email{nkajino@math.kobe-u.ac.jp}
\thanks{N.K.\ was partially supported by JSPS KAKENHI Grant Number 26287017.}



\date{October 2, 2015}

\dedicatory{}

\begin{abstract}
The \emph{Liouville Brownian motion (LBM)}, recently introduced by Garban,
Rhodes and Vargas and in a weaker form also by Berestycki,
is a diffusion process evolving in a planar random geometry induced by the
\emph{Liouville measure} $M_\gamma$, formally written as
$M_\gamma(dz)=e^{\gamma X(z)-{\gamma^2} \bbE[X(z)^2]/2}\, dz$, $\gamma\in(0,2)$,
for a (massive) Gaussian free field $X$. It is an $M_\gamma$-symmetric
diffusion defined as the time change of the two-dimensional Brownian motion
by the positive continuous additive functional with Revuz measure $M_\gamma$.

In this paper we provide a detailed analysis of the heat kernel $p_t(x,y)$
of the LBM. Specifically, we prove its joint continuity, a locally uniform
sub-Gaussian upper bound of the form
$p_t(x,y)\leq C_{1} t^{-1} \log(t^{-1})
	\exp\bigl(-C_{2}((|x-y|^{\beta}\wedge 1)/t)^{\frac{1}{\beta -1}}\bigr)$
for $t\in(0,\frac{1}{2}]$ for each $\beta>\frac{1}{2}(\gamma+2)^2$,
and an on-diagonal lower bound of the form
$p_{t}(x,x)\geq C_{3}t^{-1}\bigl(\log(t^{-1})\bigr)^{-\eta}$
for $t\in(0,t_{\eta}(x)]$, with $t_{\eta}(x)\in(0,\frac{1}{2}]$
\emph{heavily dependent on $x$}, for each $\eta>18$ for
\emph{$M_{\gamma}$-almost every}\ $x$.
As applications, we deduce that the pointwise spectral dimension equals
$2$ $M_\gamma$-a.e.\ and that the global spectral dimension is also $2$.
\end{abstract}

\maketitle

\section{Introduction}
One of the main mathematical issues in the theory of two-dimensional
Liouville quantum gravity is to construct a random geometry on a
two-dimensional manifold (say $\bbR^2$ equipped with the Euclidian metric
$dx^2$) which can be formally described by
a Riemannian metric tensor of the form
\begin{align} \label{eq:LQG}
 e^{\gamma X(x)} \, dx^2,
\end{align}
where $X$ is a massive Gaussian free field on $\bbR^2$ defined on a probability
space $(\Omega,\mathcal{A},\bbP)$ and $\gamma \in(0,2)$ is a parameter.
The study of Liouville quantum gravity is mainly motivated by the so-called KPZ-formula
(for Knizhnik, Polyakov and Zamolodchikov), which relates some geometric
quantities in a number of models in statistical physics to their formulation
in a setup governed by this random geometry. In this context, by the KPZ relation
the parameter $\gamma$  can be expressed in terms of a certain physical constant
called the central charge of the underlying model. We refer to \cite{DS11} and
to the survey article \cite{Ga13} for more details on this topic.

However, to give rigorous sense to the expression \eqref{eq:LQG} is a highly
non-trivial problem. Namely, as the correlation function of the Gaussian free
field $X$ exhibits short scale logarithmically divergent behaviour,
the field $X$ is not a function but only a random distribution.
In other words, the underlying geometry is too rough to make sense in the
classical Riemannian framework, so some regularisation is required.
While it is not clear how to execute a regularisation procedure on the level of
the metric, the method performs well enough to construct the associated
volume form. More precisely, using the theory of Gaussian multiplicative chaos
established by Kahane in \cite{Ka85} (see also \cite{RV13a}),
by a certain cutoff procedure one can define the associated volume measure
$M_\gamma$ for $\gamma \in (0,2)$, called the \emph{Liouville measure}.
It can be interpreted as being given by
\begin{align*}
 M_\gamma(A)=\int_A e^{\gamma X(z)-\tfrac{\gamma^2} 2  \bbE[X(z)^2]} dz,
\end{align*}
but this expression for $M_\gamma$ is only very formal, for $M_\gamma$ is known
to be singular with respect to the Lebesgue measure by a result
\cite[(141)]{Ka85} by Kahane (see also \cite[Theorems~4.1 and 4.2]{RV13a}).
Recently, in \cite{GRV13} Garban, Rhodes and Vargas have constructed the natural
diffusion process $\mathcal{B}=(\mathcal{B}_t)_{t\geq 0}$ associated with
\eqref{eq:LQG}, which they call the \emph{Liouville Brownian motion (LBM)}.
Similar results have been simultaneously obtained in a weaker form also by
Berestycki \cite{Be13}. On a formal level, $\mathcal{B}$ is the solution of the SDE
\begin{align*}
d\mathcal{B}_t=e^{- \tfrac \gamma 2 X(\mathcal{B}_t)+\tfrac{\gamma^2} 4  \bbE[X(\mathcal{B}_t)^2]} d\bar B_t, 
\end{align*}
where $\bar B=(\bar B_t)_{t\geq 0}$ is a standard Brownian motion on $\bbR^2$
independent of $X$. In view of the Dambis-Dubins-Schwarz theorem this SDE
representation suggests defining the LBM $\mathcal{B}$
as a time change of another planar Brownian motion $B=(B_t)_{t\geq 0}$.
This has been rigorously carried out in \cite{GRV13}, and then by general theory
the LBM turns out to be symmetric with respect to the Liouville measure $M_\gamma$.
In the companion paper \cite{GRV13a} Garban, Rhodes and Vargas also identified 
the Dirichlet form associated with $\mathcal{B}$ and they showed that
the transition semigroup is absolutely continuous with respect to $M_\gamma$,
meaning that the Liouville heat kernel $p_t(x,y)$ exists.
Moreover, they observed that the intrinsic metric $d_{\mathcal{B}}$
generated by that Dirichlet form is identically zero, which indicates that
\begin{align*}
 \lim_{t \downarrow 0} t \log p_t(x,y)=-\frac{d_{\mathcal{B}}(x,y)^2}{2}=0, \qquad x,y \in \bbR^2,
\end{align*}
and therefore some non-Gaussian heat kernel behaviour is expected. This
degeneracy of the intrinsic metric is known to occur typically for diffusions
on fractals, whose heat kernels indeed satisfy the so-called sub-Gaussian
estimates; see e.g.\ the survey articles \cite{Ba13,Ku14} and references therein.

In this paper we continue the analysis of the Liouville heat kernel,
which has been initiated simultaneously and independently in \cite{MRVZ14}.
As our first main results we obtain the continuity of the heat kernel and
a rough upper bound on it.

\begin{theorem} \label{thm:cont_hk}
Let $\gamma\in (0,2)$. Then $\bbP$-a.s.\ the following hold: A (unique) jointly
continuous version $p=p_{t}(x,y):(0,\infty)\times\bbR^2 \times\bbR^2\to[0,\infty)$
of the Liouville heat kernel exists and is $(0,\infty)$-valued,
and in particular the Liouville Brownian motion $\mathcal{B}$ is irreducible.
Moreover, the associated transition semigroup $(P_t)_{t>0}$ defined by
\begin{align*}
P_t f(x):=E_x[f(\mathcal{B}_t)]=\int_{\bbR^2} p_t(x,y) f(y) \, M_\gamma(dy),
	\qquad x\in \bbR^2,
\end{align*}
is strong Feller, i.e.\ $P_t f$ is continuous for any bounded Borel measurable
$f: \bbR^2 \rightarrow \bbR$.
\end{theorem}

\begin{theorem} \label{thm:on-diag_glob}
Let $\gamma\in (0,2)$. Then $\bbP$-a.s., for any $\beta>\frac{1}{2}(\gamma+2)^2$
and any bounded $U\subset \bbR^2$ there exist random constants
$C_i=C_i(X,\gamma,U, \beta)>0$, $i=1,2$, such that
\begin{align} \label{eq:on-diag_glob}
 p_t(x,y)=p_t(y,x) \leq C_1 t^{-1} \log(t^{-1})
	\exp\biggl(-C_2 \Bigl(\frac{|x-y|^\beta \wedge 1}{t} \Bigr)^{\frac{1}{\beta -1}}\biggr)
\end{align}
for all $t\in (0, \tfrac 1 2]$, $x \in \bbR^2$ and $y\in U$,
where $|\cdot|$ denotes the Euclidean norm on $\bbR^2$.
\end{theorem}

Since $\beta>\frac{1}{2}(\gamma+2)^2>2$, the off-diagonal part
$\exp\bigl(-C_{2}((|x-y|^{\beta}\wedge 1)/t)^{\frac{1}{\beta -1}}\bigr)$
of the bound \eqref{eq:on-diag_glob} indicates that the process diffuses
slower than the two-dimensional Brownian motion, which is why such a bound
is called \emph{sub}-Gaussian. We do not expect that the lower bound
$\frac{1}{2}(\gamma+2)^2$ for the exponent $\beta$ is best possible. Unfortunately,
Theorem~\ref{thm:on-diag_glob} alone does not even exclude the possibility
that $\beta$ could be taken arbitrarily close to $2$, which in the case of the two-dimensional
torus has been in fact disproved in a recent result \cite[Theorem~5.1]{MRVZ14}
by Maillard, Rhodes, Vargas and Zeitouni showing that $\beta$ satisfying
\eqref{eq:on-diag_glob} for small $t$ must be at least $2+\gamma^{2}/4$.
In this sense the Liouville heat kernel does behave anomalously,
which is natural to expect from the degeneracy of the intrinsic metric
associated with the LBM.

From the conformal invariance of the planar Brownian motion $B$
it is natural to expect that the LBM $\mathcal{B}$ as a time change of $B$
admits two-dimensional behaviour, as was observed by physicists in \cite{ABNRW98}
and in a weak form proved in \cite{RV13} (see Remark~\ref{rmk:spctr_dim} below).
The on-diagonal part $t^{-1}\log(t^{-1})$ in \eqref{eq:on-diag_glob} shows a
sharp upper bound in this spirit except for a logarithmic correction, and
we will also prove the following on-diagonal
lower bound valid for \emph{$M_\gamma$-a.e.}\ $x \in \bbR^2$, which matches
\eqref{eq:on-diag_glob} besides another logarithmic correction.

\begin{theorem}\label{thm:on-diag_low}
Let $\gamma\in (0,2)$. Then $\bbP$-a.s., for $M_\gamma$-a.e.\ $x\in \bbR^2$,
for any $\eta>18$ there exist random constants $C_3=C_3(X,\gamma,|x|,\eta)>0$
and $t_0(x)=t_0(X,\gamma,\eta,x)\in (0,\frac{1}{2}]$ such that
\begin{align} \label{eq:on-diag_low}
p_t(x,x)\geq C_3 t^{-1} \bigl(\log(t^{-1})\bigr)^{-\eta}, \qquad \forall t\in (0,t_0(x)].
\end{align}
\end{theorem}

Combining the on-diagonal estimates in Theorems~\ref{thm:on-diag_glob} and \ref{thm:on-diag_low},
we can immediately identify the pointwise spectral dimension as $2$.

\begin{corro} \label{thm:spctr_dim}
 Let $\gamma\in (0,2)$. Then $\bbP$-a.s., for $M_\gamma$-a.e.\ $x\in \bbR^2$,
\begin{align*}
 \lim_{t\downarrow 0} \frac{2 \log p_t(x,x)}{-\log t}=2.
\end{align*}
\end{corro}

Essentially from Theorems~\ref{thm:on-diag_glob} and \ref{thm:on-diag_low}
we shall further deduce that the global spectral dimension, that is the
growth order of the Dirichlet eigenvalues of the generator on bounded
open sets, is also $2$; see Subsection~\ref{sec:spctr_dim} for details.

\begin{remark} \label{rmk:spctr_dim}
 In \cite[Theorem~3.6]{RV13} the following result on the spectral dimension
has been proved: $\bbP$-a.s., for any $\alpha>0$ and for all $x\in \bbR^2$, 
\begin{align} \label{eq:sp_dim_int1}
\lim_{y\to x} \int_0^\infty e^{-\lambda t} t^\alpha p_t(x,y) \, dt <\infty, \qquad \forall \lambda >0,
 \end{align}
and 
 \begin{align}\label{eq:sp_dim_int2}
\lim_{y\to x} \int_0^\infty e^{-\lambda t} p_t(x,y) \, dt =\infty, \qquad \forall \lambda >0.
 \end{align}
In \cite{RV13} the left hand sides were interpreted as the integrals in $t$
of the on-diagonal heat kernel $p_t(x,x)$, which was needed due to the lack
of the knowledge of the continuity of $p_t(x,y)$.  
By Theorem~\ref{thm:cont_hk} this interpretation can be made rigorous now,
and moreover, \eqref{eq:sp_dim_int1}  follows immediately from
Theorem~\ref{thm:on-diag_glob}. On the other hand, \eqref{eq:sp_dim_int2}
is actually an easy consequence of the Dirichlet form theory.
Indeed,  by \cite[Exercises~2.2.2 and  4.2.2]{FOT11}
$\int_0^\infty e^{-\lambda t} p_t(x,x) \, dt$ is equal to the reciprocal 
of the $\lambda$-order capacity of the singleton $\{x\}$ with respect to
the LBM, and this capacity is zero by \cite[Lemma~6.2.4 (i)]{FOT11} and
the fact that the same holds for the planar Brownian motion.
\end{remark}

The proofs of our main results above are mainly based on the moment estimates
for the Liouville measure $M_{\gamma}$ by \cite{Ka85,RV10} and those for the
exit times of the LBM $\mathcal{B}$ from balls by \cite{GRV13}, together with
the general fact from time change theory that the Green operator of the LBM
has exactly the same integral kernel as that of the planar Brownian motion
(see \eqref{eq:defG_U} below). To turn those moment estimates into
$\bbP$-almost sure statements, we need some Borel-Cantelli arguments that
cannot provide us with uniform control on various random constants over unbounded sets.
For this reason we can expect the estimate \eqref{eq:on-diag_glob} to hold
only \emph{locally} uniformly, so that in Theorem~\ref{thm:on-diag_glob}
we cannot drop the dependence of the constants
$C_{1},C_{2}$ on $U$  or the cutoff of $|x-y|$ at $1$ in the exponential.
Also to remove the logarithmic corrections in \eqref{eq:on-diag_glob}
and \eqref{eq:on-diag_low} and the restriction to $M_\gamma$-a.e.\ points in
Theorem~\ref{thm:on-diag_low} and Corollary~\ref{thm:spctr_dim} one would need
to have good uniform control on the ratios of the $M_\gamma$-measures of
concentric balls with different radii. However, we cannot hope for such control
in view of \cite[Remark~A.2]{BGRV14}, where it is claimed that
\begin{align*}
\limsup_{r\downarrow 0}\sup_{x\in B(0,1)}
	\frac{M_{\gamma}\bigl(B(x,2r)\bigr)}{M_{\gamma}\bigl(B(x,r)\bigr)^{1-\eta}}=\infty
\end{align*}
for any $\eta\in\bigl(-\infty,\frac{\gamma^{2}}{4+\gamma^{2}}\bigr)$ $\bbP$-a.s.,
with $B(x,R):=\{y \in \bbR^2:\, |x-y| < R \}$ for $x\in \bbR^2$ and $R>0$.

The LBM can also be constructed on other domains like the torus,
the sphere or planar domains $D\subset \bbR^2$ equipped with a log-correlated
Gaussian field like the (massive or massless) Gaussian free field
(cf.\ \cite[Section~2.9]{GRV13}). In fact, Theorem~\ref{thm:cont_hk} has been
simultaneously and independently obtained in \cite{MRVZ14} for the LBM
on the torus, where thanks to the boundedness of the space one can 
utilise the eigenfunction expansion of the heat kernel to 
prove its continuity and the strong Feller property of the semigroup.
On the other hand, in our case of $\bbR^2$ the Liouville heat kernel
$p_t(x,y)$ does not admit such an eigenfunction expansion and the proof
of its continuity and the strong Feller property requires some additional
arguments.  Therefore, although the proofs of our results should directly
transfer to the other domains mentioned above, we have decided to work on
the plane $\bbR^2$ in this paper for the sake of simplicity and in order
to stress that our methods also apply to the case of unbounded domains.

In \cite{MRVZ14} Maillard, Rhodes, Vargas and Zeitouni have also obtained upper
and lower estimates of the Liouville heat kernel on the torus. Their heat kernel
upper bound in \cite[Theorem~4.2]{MRVZ14} involves an on-diagonal part of the
form $C  t^{-(1+\delta)}$ for any $\delta>0$ and an off-diagonal part
of the form $\exp\bigl(-C (|x-y|^\beta /t )^{\frac{1}{\beta -1}}\bigr)$
for any $\beta>\beta_0(\gamma)$, where $\beta_0(\gamma)$ is a constant larger
than our lower bound $\frac{1}{2}(\gamma+2)^2$ on the exponent $\beta$ and
satisfies $\lim_{\gamma \uparrow 2}\beta_0(\gamma)=\infty$. Thus
Theorem~\ref{thm:on-diag_glob} gives a better estimate, and we prove it
by self-contained, purely analytic arguments while the proof in \cite{MRVZ14}
relies on \eqref{eq:sp_dim_int1}, whose proof in \cite{RV13} is technically
involved. Concerning lower bounds, an on-diagonal lower bound as in
Theorem~\ref{thm:on-diag_low} is not treated in \cite{MRVZ14}.
On the other hand, their off-diagonal lower bound \cite[Theorem~5.1]{MRVZ14},
which implies the bound $\beta\geq 2+\gamma^{2}/4$ for any such exponent
$\beta$ as in \eqref{eq:on-diag_glob} (in the case of the torus) as mentioned
above after Theorem~\ref{thm:on-diag_glob}, is not covered by our results.

The rest of the paper is organised as follows. In Section~\ref{sec:LBM} we recall
the construction of the LBM in \cite{GRV13} and introduce the precise setup. 
In Section~\ref{sec:prelim} we prove preliminary estimates on the volume decay
of the Liouville measure and on the exit times from balls needed in the proofs.
In Section~\ref{sec:feller} we show that the resolvent operators of the LBM
killed upon exiting an open set have the strong Feller property,
which is needed in Section~\ref{sec:cont_ub} to prove
Theorems~\ref{thm:cont_hk} and \ref{thm:on-diag_glob}.
In Subsection~\ref{sec:Ubounded} we show the continuity of the Dirichlet
heat kernel associated with the killed LBM on a bounded open set
by using its eigenfunction expansion, and in Subsection~\ref{sec:Uunbounded}
we then deduce the continuity of the heat kernel and the strong Feller property
on unbounded open sets, as well as Theorem~\ref{thm:on-diag_glob},
using a recent result in \cite{GK14}.
Finally, in Section~\ref{sec:lb} we show the on-diagonal lower bound in
Theorem~\ref{thm:on-diag_low} and thereby identify the pointwise and
global spectral dimensions as $2$. 

Throughout the paper, we write $C$ for random positive constants depending
on the realisation of the field $X$, which may change on each appearance,
whereas  the numbered random positive constants $C_i$ will be kept the same.
Analogously, non-random positive constants will be denoted by $c$ or $c_i$,
respectively. The symbols $\subset$ and $\supset$ for set inclusion \emph{allow}
the case of the equality. We denote by $|\cdot|$ the Euclidean norm on $\bbR^2$
and by $B(x,R):=\{y \in \bbR^2:\, |x-y| < R \}$, $x\in \bbR^2$, $R>0$, open
Euclidean balls in $\bbR^2$ and for abbreviation we set $B(R):=B(0,R)$.
Lastly, for non-empty $U\subset \bbR^2$ and $f:U\rightarrow \bbR$ we write
$\| f\|_{\infty}:=\| f\|_{\infty,U}:= \sup_{x\in U} |f(x)|$.

\section{Liouville Brownian motion} \label{sec:LBM}
\subsection{Massive Gaussian free field and Liouville measure}
Consider a massive Gaussian free field $X$ on the whole plane $\bbR^2$, i.e.\ a
Gaussian Hilbert space associated with the Sobolev space $\mathcal{H}^1_m$
defined as the closure of $C_c^\infty(\bbR^2)$ with respect to the inner product
\begin{align*}
 \langle f,g\rangle_m:= m^2 \int_{\bbR^2} f(x) \, g(x) \, dx + \int_{\bbR^2} \nabla f(x) \cdot \nabla g(x) \, dx,
\end{align*}
where $m>0$ is a parameter called the mass. More precisely,
$(\langle X, f\rangle_m)_{f\in \mathcal{H}^1_m}$ is a family of Gaussian random
variables on a probability space $(\Omega, \mathcal{A}, \bbP)$ with mean $0$
and covariance
\begin{align*}
 \bbE\bigl[\langle X, f\rangle_m  \langle X, g\rangle_m\bigr]= 2\pi \langle  f,g \rangle_m.
\end{align*}
In other words, the covariance function of $X$ is given by the massive Green function
$g^{(m)}$ associated with the operator $m^2-\Delta$, which can be written as
\begin{align} \label{eq:mGreenFunc}
 g^{(m)}(x,y)=\int_0^\infty \frac{1}{2u} e^{-\frac{m^2}{2}u-\frac{|x-y|^2}{2u}} \, du
	= \int_1^\infty \frac{k^{(m)}\bigl(u(x-y)\bigr)}{u} \, du
\end{align}
with
\begin{align*}
 k^{(m)}(z):= \frac 1 2 \int_0^\infty e^{-\frac{m^2}{2v}|z|^2-\frac v 2} \, dv.
\end{align*}
Following \cite{GRV13} we now introduce an $n$-regularised version of $X$.
To that aim let $(a_n)_{n\geq 0}\subset\mathbb{R}$ be an unbounded strictly
increasing sequence with $a_0=1$ and let $(Y_n)_{n\geq 1}$ be a family of
independent \emph{continuous} Gaussian fields on $\bbR^2$ defined also on
$(\Omega, \mathcal{A}, \bbP)$ with mean $0$ and covariance
\begin{align} \label{eq:def_kn}
 \bbE\bigl[ Y_n(x) Y_n(y)\bigr]
	=\int_{a_{n-1}}^{a_{n}} \frac{k^{(m)}\bigl(u(x-y)\bigr)}{u} \, du
	=:g^{(m)}_n(x,y);
\end{align}
here, such $Y_n$ can be constructed by applying a version \cite[Problem~2.2.9]{KS91}
of the Kolmogorov-\v{C}entsov continuity theorem to a Gaussian field on $\bbR^2$
with mean $0$ and covariance $g^{(m)}_n$, which in turn exists by the Kolmogorov
extension theorem (see e.g.\ \cite[Theorems~12.1.2 and 12.1.3]{Du02})
since $\bigl(g^{(m)}_n(x,y)\bigr)_{x,y\in \Xi}$ is a non-negative definite
real symmetric matrix for any finite $\Xi\subset\bbR^2$.
Then for each $n\geq 1$, the $n$-regularised field $X_n$ is defined as
\begin{align*}
 X_n(x):=\sum_{k=1}^n Y_k(x),\qquad x\in \bbR^2,
\end{align*}
and the associated random Radon measure $M_n=M_{\gamma,n}$ on $\bbR^2$ is given by
\begin{align} \label{eq:dfn_Mn}
 M_n(dx):=\exp\Bigl(\gamma X_n(x)-\tfrac {\gamma^2} 2 \bbE\bigl[X_n(x)^2\bigr]\Bigr)\,dx
\end{align}
with a parameter $\gamma \geq 0$. By the classical theory of
Gaussian multiplicative chaos established in Kahane's seminal work \cite{Ka85}
(see also \cite{RV13a}) we have the following:
$\bbP$-a.s.\ the family $(M_n)_{n\geq 1}$ converges vaguely on $\bbR^2$
to a random Radon measure $M=M_\gamma$ called the \emph{Liouville measure},
whose law is uniquely determined by $\gamma$ and the covariance function
$g^{(m)}$ of $X$, and $M$ has full support $\bbP$-a.s.\ for $\gamma \in [0,2)$
and is identically zero $\bbP$-a.s.\ for $\gamma\geq 2$. Throughout the rest
of this paper, we assume that $\gamma \in (0,2)$ is fixed and we will drop it
from our notation, although the quantities defined through the Liouville
measure $M=M_{\gamma}$ will certainly depend on $\gamma$.

\subsection{Definition of Liouville Brownian Motion}
The Liouville Brownian motion has been constructed by Garban, Rhodes and Vargas
in \cite{GRV13} as the canonical diffusion process under the geometry induced
by the measure $M$. More precisely, they have constructed
a positive continuous additive functional $F=\{F_t\}_{t\geq 0}$ of the planar
Brownian motion $B$ naturally associated with the measure $M$ and they have
defined the LBM as $\mathcal{B}_t=B_{F^{-1}_t}$.
In this subsection we briefly recall the construction.

Let $\Omega':=C([0,\infty),\bbR^2)$, let $B=(B_t)_{t\geq 0}$ be the coordinate
process on $\Omega'$ and set $\mathcal{G}^0_\infty:=\sigma(B_s; \, s<\infty)$
and $\mathcal{G}^0_t:=\sigma(B_s; \, s\leq t)$, $t\geq 0$. Let
$\{P_x\}_{x\in \bbR^2}$ be the family of probability measures on
$(\Omega',\mathcal{G}^0_\infty)$ such that for each $x\in \bbR^2$,
$B=(B_t)_{t\geq 0}$ under $P_x$ is a two-dimensional Brownian motion starting
at $x$. We denote by $\{ \mathcal{G}_t\}_{t\in [0,\infty]}$ the minimum completed
admissible filtration for $B$ with respect to $\{P_x\}_{x\in \bbR^2}$
as defined e.g.\ in \cite[Section~A.2]{FOT11}.
Moreover, let $\{\theta_t\}_{t\geq 0}$ be the family of shift mappings on
$\Omega'$, i.e.\ $B_{t+s}=B_t \circ \theta_s$, $s,t \geq 0$. Finally, we write
$q_t(x,y):=(2\pi t)^{-1} \exp\bigl(-|x-y|^2/(2t)\bigr)$, $t>0$, $x,y\in \bbR^2$,
for the heat kernel associated with $B$.

\begin{definition}
i) A $[-\infty,\infty]$-valued stochastic process $A=(A_t)_{t\geq 0}$
on $(\Omega',\mathcal{G}_{\infty})$
is called a \emph{positive continuous additive functional (PCAF)} of $B$
\emph{in the strict sense}, if $A_t$ is $\mathcal{G}_t$-measurable for every
$t\geq 0$ and if there exists a set $\Lambda \in \mathcal{G}_\infty$,
called a \emph{defining set} for $A$, such that
\begin{enumerate}
\item [a)] for all $x\in \bbR^2$, $P_x[\Lambda]=1$,
\item [b)] for all $t\geq 0$, $\theta_t(\Lambda)\subset \Lambda$,
\item [c)] for all $\om \in \Lambda$,
$[0,\infty)\ni t\mapsto A_t(\om)$ is a $[0,\infty)$-valued continuous function
with $A_0(\om)=0$ and
\begin{align*}
 A_{t+s}(\om)=A_t(\om)+ A_s \circ \theta_t (\om), \qquad \forall s,t\geq 0.
\end{align*}
\end{enumerate}

ii) Two such functionals $A^1$ and $A^2$ are called \emph{equivalent}
if $P_x[A^1_t=A^2_t]=1$ for all $t>0$, $x\in \bbR^2$, or equivalently,
there exists $\Lambda \in \mathcal{G}_\infty$ which is a defining set for
both $A^1$ and $A^2$ such that $A^1_t(\om)=A^2_t(\om)$ for all $t\geq 0$,
$\om\in \Lambda$. Equivalent PCAFs in the strict sense will always
be identified hereafter.

iii) For any such $A$, a Borel measure $\mu_A$ on $\bbR^2$ satisfying 
\begin{align*}
 \int_{\bbR^2} f(y) \, \mu_A(dy)
	=\lim_{t\downarrow 0} \frac 1 t \int_{\bbR^2} E_x \Bigl[ \int_0^t f(B_s) \, dA_s \Bigr] \, dx
\end{align*}
for any non-negative Borel function $f:\bbR^2\to[0,\infty]$ is called the
\emph{Revuz measure} of $A$, which exists uniquely by general theory
(see e.g.\ \cite[Theorem~A.3.5]{CF12}).
\end{definition}

For every $n\in \bbN$ let now
$F^n_t: \Omega \times \Omega' \rightarrow [0,\infty)$ be defined as
\begin{align} \label{eq:dfn_Fn}
 F^n_t:=\int_0^t
	\exp\Bigl(\gamma X_n(B_s)- \tfrac{\gamma^2} 2 \bbE\bigl[X_n(B_s)^2\bigr]\Bigr)
	\, ds, \qquad t\geq 0,
\end{align}
which is strictly increasing in $t$.
Note that for every $n$ the functional $F^n=(F^n_t)_{t\geq 0}$ considered
as a process defined on $(\Omega',\mathcal{G}^{0}_{\infty})$ is a PCAF of
$B$ in the strict sense with defining set $\Omega'$ and Revuz measure $M_n$.


\begin{theorem}[{\cite[Theorem~2.7]{GRV13}}] \label{thm:GRV13}
$\bbP$-a.s.\ the following hold:
\begin{enumerate}
 \item[i)] There exists a unique PCAF $F$ in the strict sense
whose Revuz measure is $M$.

\item[ii)] For all $x\in \bbR^2$, $P_x$-a.s., $F$ is strictly increasing and
satisfies $\lim_{t \to \infty} F_t=\infty$.

\item[iii)]  For all $x\in \bbR^2$, $F^n$ converges to $F$ in $P_x$-probability
in the space $C([0,\infty),\bbR)$ equipped with the topology of uniform
convergence on compact sets.
\end{enumerate}
The process $(\mathcal{B},\{P_x\}_{x\in\bbR^2})$, $\bbP$-a.s.\ defined by
$\mathcal{B}_t:=B_{F^{-1}_t}$, $t\geq 0$, is called the (massive)
\emph{Liouville Brownian motion (LBM)}.
\end{theorem}

Thanks to Theorem~\ref{thm:GRV13}, we can apply the general theory of time
changes of Markov processes to have the following properties of the LBM:
First, it is a recurrent diffusion on $\bbR^2$ by \cite[Theorems~A.2.12 and 6.2.3]{FOT11}.
Furthermore by \cite[Theorem~6.2.1 (i)]{FOT11} (see also \cite[Theorem~2.18]{GRV13}),
the LBM is $M$-symmetric, i.e.\ its transition semigroup $(P_t)_{t>0}$ given by
\begin{align*}
P_t (x,A):=E_x[\mathcal{B}_t \in A]
\end{align*}
for $t\in(0,\infty)$, $x\in\bbR^{2}$ and a Borel set $A\subset\bbR^{2}$,
satisfies
\begin{align*}
 \int_{\bbR^2} P_t f \cdot g \, dM = \int_{\bbR^2} f \cdot P_t g \, dM 
\end{align*}
for all Borel measurable functions $f,g:\ \bbR^2 \rightarrow [0,\infty]$.
Here the Borel measurability of $P_t(\cdot, A)$ can be deduced from
\cite[Corollary~2.20]{GRV13} (or from Proposition~\ref{prop:pcaf} below). 

\begin{remark}
\cite[Corollary~2.20]{GRV13} states that $(P_t)_{t>0}$ is a Feller semigroup,
meaning that $P_t$ preserves the space of bounded continuous functions.
Note that this is different from the notion of a Feller semigroup as for
instance in \cite{CF12,FOT11}, i.e.\ a strongly continuous Markovian semigroup
on the space of continuous functions vanishing at infinity.
It is not known whether $(P_t)_{t>0}$ is a Feller semigroup in the latter sense.
\end{remark}

It is natural to expect that the LBM can be constructed in such a way that it
depends measurably on the randomness of the field $X$. However, this measurability
does not seem obvious from the construction in \cite{GRV13}, since there the
existence of the PCAF $F$ has been deduced from some general theory on the
Revuz correspondence for $\bbP$-a.e.\ fixed realisation of $M$. To overcome
this issue, in the following proposition we show for $\bbP$-a.e.\ environment
the pathwise convergence of $F^n$ towards $F$ in an appropriate
$\{P_x\}_{x\in \bbR^2}$-a.s.\ sense which also ensures the measurability
of $F_t$ and $\mathcal{B}_t$ with respect to the product $\sigma$-field
$\mathcal{A}\otimes \mathcal{G}^0_\infty$ for all $t\geq 0$.
The proof is given in Appendix~\ref{app:pcaf}.

\begin{prop} \label{prop:pcaf}
There exists a set  $\Lambda \in \mathcal{A}\otimes \mathcal{G}_\infty^0$
such that the following hold:

\begin{enumerate}

 \item[i)] For $\bbP$-a.e.\ $\om\in \Omega$, $P_x[\Lambda^\om]=1$ for any $x\in \bbR^2$,
	where $\Lambda^\om:=\{ \om'\in\Omega':\,  (\om,\om')\in \Lambda\}$.

 \item[ii)] For every $(\om,\om')\in \Lambda$ the following limits exist
	in $\bbR$ for all $0<s\leq t$:
	\begin{align*}
	 F_{s,t}(\om,\om')&:=\lim_{n\to \infty} \bigl(F^n_t(\om,\om')-F^n_s(\om,\om')\bigr),   \\
	 F_t(\om,\om')&:= \lim_{u\downarrow 0} F_{u,t}(\om,\om').
	\end{align*}
	Moreover, with $F_0(\om,\om'):=0$, $[0,\infty)\ni t\mapsto F_t(\om,\om')\in[0,\infty)$
	is continuous, strictly increasing and satisfies $\lim_{t\to \infty} F_t(\om,\om')=\infty$.

 \item[iii)] Let $t\geq 0$ and set $F_t:=t$ on $\Lambda^c$.
	Then $F_t$ is $\mathcal{A} \otimes \mathcal{G}_\infty^0$-measurable.

 \item [iv)] For $\bbP$-a.e.\ $\om\in\Omega$, the process
	$\bigl(F_t(\om,\cdot)\bigr)_{t\geq 0}$ is a PCAF of $B$ in the strict sense
	with defining set $\Lambda^\om$. 
\end{enumerate}

\end{prop}

The previous proposition implies easily that $F$ indeed has the Revuz measure $M$.
More strongly, we have the following proposition valid for any starting point
$x\in\bbR^2$ $\bbP$-a.s., which we prove in Appendix~\ref{app:revuz}
in a slightly more general setting for later use.

\begin{prop}  \label{prop:revuz}
$\bbP$-a.s., for all $x\in \bbR^2$ and all Borel measurable functions
$\eta: [0,\infty)\to [0,\infty]$ and $f: \bbR^2 \to [0,\infty]$,
\begin{align*} 
 E_x\Bigl[ \int_0^{\infty} \eta(t) f(B_t) \, dF_t \Bigr]
	=\int_0^\infty \int_{\bbR^2} \eta(t) f(y) q_t(x,y) \, M(dy) \, dt,
\end{align*}
and in particular, for any $t>0$,
\begin{align*}
 \int_{\bbR^2} f(y) \, M(dy)= \frac 1 t \int_{\bbR^2} E_x \Bigl[ \int_0^t f(B_s) \, dF_s \Bigr] \, dx.
\end{align*}

\end{prop}

\subsection{The Liouville Dirichlet form}
By virtue of Propositions~\ref{prop:pcaf} and \ref{prop:revuz},
we can apply the general theory of Dirichlet forms to obtain an explicit
description of the Dirichlet form associated with the LBM,
as it has been done in \cite{GRV13, GRV13a}.

Denote by $H^1(\bbR^2)$ the standard Sobolev space, that is
\begin{align*}
 H^1(\bbR^2)=\{ f\in L^2(\bbR^2, dx):\, \nabla f \in L^2(\bbR^2, dx) \},
\end{align*}
on which we define the form
\begin{align} \label{eq:def_dirform}
 \mathcal{E}(f,g)=\frac 1 2 \int_{\bbR^2} \nabla f \cdot \nabla g \, dx.
\end{align}
Recall that $(\mathcal{E}, H^1(\bbR^2))$ is the Dirichlet form of the planar
Brownian motion $B$. By $H^1_e(\bbR^2)$ we denote the extended Dirichlet space,
that is the set of $dx$-equivalence classes of Borel measurable functions $f$
on $\bbR^2$ such that $\lim_{n\to \infty} f_n=f \in \bbR$ $dx$-a.e.\ for some
$(f_n)_{n\geq 1} \subset H^1(\bbR^2)$ satisfying
$\lim_{k,l \to \infty}\mathcal{E}(f_k-f_l, f_k-f_l)=0$.
By \cite[Theorem~2.2.13]{CF12} we have the following identification of $H^1_e(\bbR^2)$:
\begin{align*}
 H^1_e(\bbR^2)=\{ f \in L^2_{\mathrm{loc}}(\bbR^2, dx): \nabla f \in L^2(\bbR^2, dx) \}. 
\end{align*}
The \emph{capacity} of a set $A\subset \bbR^2$ is defined by
\begin{align*}
 \mathrm{Cap}(A)=\inf_{\substack{B \subset \bbR^2 \, \text{open} \\ A \subset B}}
	\inf_{\substack{f \in H^1(\bbR^2) \\ f|_{B}\geq 1 \, \text{$dx$-a.e.} }}
	\Bigl\{\mathcal{E}(f,f)+\int_{\bbR^2} f^2 \, dx \Bigr\}.
\end{align*}
A set $A\subset \bbR^2$ is called \emph{polar} if $\mathrm{Cap}(A)=0$.
We call a function $f$ \emph{quasi-continuous} if for any $\varepsilon >0$
there exists an open $U\subset \bbR^2$ with $\mathrm{Cap}(U) < \varepsilon$
such that $f|_{\bbR^2 \setminus U}$ is real-valued and continuous.
By \cite[Theorem~2.1.7]{FOT11} any $f\in H^1_e(\bbR^2)$ admits a
quasi-continuous $dx$-version $\widetilde{f}$,
which is unique up to polar sets by \cite[Lemma~2.1.4]{FOT11}.

Then, as the Liouville measure $M$ is a Radon measure on $\bbR^2$
and does not charge polar sets by \cite[Theorem~2.2]{GRV13}
(or by Propositions~\ref{prop:pcaf}, \ref{prop:revuz} and \cite[Theorem~4.1.1 (i)]{CF12}),
the Dirichlet form $(\mathcal{E},\mathcal{F})$ of the LBM $\mathcal{B}$ is a
strongly local regular symmetric Dirichlet form on $L^2(\bbR^2,M)$ which takes on
the following explicit form by \cite[Theorem~6.2.1]{FOT11}: The domain is given by
\begin{align*}
 \mathcal{F}=\bigl\{ u \in L^2(\bbR^2,M):\, u= \widetilde{f}
	\text{ $M$-a.e.\ for some  } f\in H^1_e(\bbR^2) \bigr\},
\end{align*}
which can be identified with
$\bigl\{f \in H^1_e(\bbR^2): \, \widetilde{f} \in L^2(\bbR^2,M) \bigr\}$
by \cite[Lemma~6.2.1]{FOT11}, and for $f,g\in \mathcal{F}$ the form
$\mathcal{E}(f,g)$ is given by \eqref{eq:def_dirform}.

\subsection{The killed Liouville Brownian motion}

Let $U$ be a non-empty open subset  of $\bbR^2$ and let $U\cup \{ \partial_U \}$
be its one-point compactification. We denote by
$T_U:=\inf \{ s\geq 0: \, B_s \not \in U\}$ the exit time of the
Brownian motion $B$ from $U$ and by
$\tau_U:=\inf \{ s\geq 0: \, \mathcal{B}_s \not \in U\}$
that of the LBM $\mathcal{B}$, where $\inf\emptyset:=\infty$.
Since by definition $\mathcal{B}_t=B_{F^{-1}_t}$, $t\geq 0$, and
$F$ is a homeomorphism on $[0,\infty)$, we have $\tau_U=F_{T_U}$. Let now
$B^U=(B^U_t)_{t\geq 0}$ and $\mathcal{B}^U=(\mathcal{B}^U_t)_{t\geq 0}$
denote the Brownian motion and the LBM, respectively,
killed upon exiting $U$. That is, they are diffusions on $U$ defined by
\begin{align*}
 B^U_t:= \begin{cases}
B_t & \text{if $t< T_U$,} \\
\partial_U & \text{if $t\geq T_U$,}
 \end{cases}
\qquad
 \mathcal{B}^U_t:= \begin{cases}
\mathcal{B}_t & \text{if $t< \tau_U$,} \\
\partial_U & \text{if $t\geq \tau_U$.}
 \end{cases}
\end{align*}
Then for $t,\lambda \in (0,\infty)$, the semigroup operator $P^U_t$ and the
resolvent operator $R^U_\lambda$ associated with the killed LBM $\mathcal{B}^U$
are expressed as, for each Borel function $f: U\rightarrow[-\infty,\infty]$ and
with the convention $f(\partial_U):=0$,
\begin{align*}
 P^U_t f(x):= E_x\bigl[f(\mathcal{B}^U_t)\bigr]
	\quad \text{and} \quad
	R^U_\lambda f(x):=E_x\Bigl[\int_0^{\tau_U} e^{-\lambda t} f(\mathcal{B}_t) \, dt \Bigr],
	\qquad x\in\bbR^2,
\end{align*}
provided the integrals exist. If $U$ is bounded, as a time change of $B^U$ the
killed LBM $\mathcal{B}^U$ has the same integral kernel for its Green operator
$G_U$ as $B^U$, namely for any non-negative Borel function
$f: U\rightarrow [0,\infty]$ and $x\in\bbR^2$,
\begin{align} \label{eq:defG_U}
 G_Uf(x):= E_x\Bigl[\int_0^{\tau_U} f(\mathcal{B}_t) \, dt \Bigr]
	= E_x\Bigl[\int_0^{T_U} f(B_t) \, dF_t \Bigr]
	= \int_U g_U(x,y) f(y) \, M(dy)
\end{align}
(cf.\ Proposition~\ref{prop:corrF_M}).
Here $g_U$ denotes the Euclidean Green kernel given by
\begin{align} \label{eq:killed_green_U}
 g_U(x,y)=\int_0^{\infty} q^U_t(x,y) \, dt, \qquad x,y\in \bbR^2,
\end{align}
for the heat kernel $q^U_t(x,y)$ of $B^U$:
$q^U=q^U_t(x,y):(0,\infty)\times U\times U\to[0,\infty)$ is the jointly
continuous function such that $P_x[B^U_t\in dy]=q^U_t(x,y)\,dy$ for $t>0$ and
$x\in U$, and we set $q^U_t(x,y):=0$ for $t>0$ and $(x,y)\in(U\times U)^{c}$.
Finally, we recall (see e.g.\ \cite[Example~1.5.1]{FOT11}) that
the Green function $g_{B(x_0,R)}$ over a ball $B(x_0,R)$ is of the form
\begin{align} \label{eq:killed_green}
 g_{B(x_0,R)}(x,y)=\frac 1 \pi \log \frac 1 {|x-y|}+\Psi_{x_0,R}(x,y)
	, \qquad x,y\in B(x_0,R),
\end{align}
for some continuous function $\Psi_{x_0,R}:B(x_0,R)\times B(x_0,R)\to\bbR$.

\section{Preliminary estimates} \label{sec:prelim}

\subsection{Volume decay estimates}
For our analysis of the Liouville heat kernel some good control on the volume
of small balls under the Liouville measure is needed. An upper estimate has
already been established in \cite{GRV13}, and we provide a similar lower bound
in the next lemma. The argument is based on some bounds on the negative moments
of the measure of small balls. Such bounds have been proved in \cite{RV10} in
the case where the limiting random measure is obtained through approximation
of the covariance kernel of the Gaussian free field by convolution. Since it is
not clear to the authors whether the cutoff procedure producing the approximating
measures $M_n$ is covered by the results in \cite{RV10}, we give a comparison
argument in Lemma~\ref{lem:neg_mom}.

In the rest of this section, we write
$\tilde \xi(q):=(2+\frac{\gamma^{2}}{2})q+\frac{\gamma^2}{2}q^2$ for $q>0$.

\begin{lemma} \label{lem:vol_dec}
 Let $\alpha_1:=\frac{1}{2}(\gamma+2)^2$ and $\alpha_2:=\frac{1}{2}(2-\gamma)^2$.
Then $\bbP$-a.s., for any $\varepsilon>0$ and any $R\geq 1$ there exist
$C_i=C_i(X, \gamma, R, \varepsilon)>0$, $i=4,5$, such that
\begin{align} \label{eq:voldec}
 C_4 r^{\alpha_1+\varepsilon}\leq M\bigl(B(x,r)\bigr) \leq C_5 r^{\alpha_2-\varepsilon},
	\qquad \forall x\in B(R), \, r\in(0,1].
\end{align}
\end{lemma}
\begin{proof}
By the monotonicity of \eqref{eq:voldec} in $\varepsilon$ and $R$ it suffices
to show \eqref{eq:voldec} $\bbP$-a.s.\ for each $\varepsilon$ and $R$.
The upper bound is proved in \cite[Theorem~2.2]{GRV13}.
We show the lower bound in the same manner. Let $q:=2/\gamma$, so that
$\alpha_1= (2+ \tilde \xi (q))/q$. Let $\varepsilon>0$ and $R\geq 1$ be fixed,
and for $n\geq 1$ we set $r_n:=2^{-n}R$ and
\begin{align} \label{def:grid}
 \Xi_{R,n}:=\bigl\{ \bigl(\tfrac k {2^{n}}R, \tfrac l {2^{n}}R \bigr) : \, k,l\in \bbZ,\ |k|,|l| \leq 2^{n} \bigr\} \subset [-R,R]^2.
\end{align}
Then for each $n\geq 1$, by \v{C}eby\v{s}ev's inequality
and Lemma~\ref{lem:neg_mom},
\begin{align*}
 &\bbP\Bigl[\min_{x\in \Xi_{R,n}} M\bigl(B(x,r_n)\bigr) \leq 2^{-n(\alpha_1+\varepsilon)} \Bigr] \\
 &\mspace{23mu}	=\bbP\Bigl[\max_{x\in \Xi_{R,n}} M\bigl(B(x,r_n)\bigr)^{-q} \geq 2^{n(\alpha_1+\varepsilon)q}\Bigr]
	\leq \sum_{x\in \Xi_{R,n}} \bbP\Bigl[ M\bigl(B(x,r_n)\bigr)^{-q} \geq 2^{n(\alpha_1+\varepsilon)q}\Bigr] \\
 &\mspace{23mu}\leq 2^{-n(\alpha_1+\varepsilon)q} \sum_{x\in \Xi_{R,n}}\bbE\Bigl[ M\bigl(B(x,r_n)\bigr)^{-q}\Bigr]
	\leq 2^{-n(\alpha_1+\varepsilon)q} 2^{2n+3}c2^{n\tilde{\xi}(q)}
	=c 2^{-\varepsilon q n}
\end{align*}
for some $c=c(\gamma,R)>0$. Thus
$\sum_{n=1}^\infty \bbP\bigl[\min_{x\in \Xi_{R,n}} M\bigl(B(x,r_n)\bigr)
	\leq 2^{-n(\alpha_1+\varepsilon)} \bigr] <\infty$,
so that by the Borel-Cantelli lemma $\bbP$-a.s.\ for some $C=C(X,\gamma,R,\varepsilon)>0$
we have that $M\bigl(B(x,r_n)\bigr)\geq C 2^{-n (\alpha_1+\varepsilon)}$
for all $n\geq 1$ and  all $x\in \Xi_{R,n}$.
Since for every $y\in B(R)$ and $r\in(0,1]$ we have $B(y,r)\supset B(x,r_n)$
for some $x\in \Xi_{R,n}$ with $n$ satisfying
$\frac{1}{4}r\leq r_{n}<\frac{1}{2}r$, the claim follows.
\end{proof}

\subsection{Exit time estimates}
In this subsection we provide some lower estimates on the exit times from balls
which are needed in the proof of Theorems~\ref{thm:cont_hk} and \ref{thm:on-diag_glob}.
More precisely, we establish estimates on the tail behaviour at zero of these exit
times by showing certain $\bbP$-a.s.\ local uniform bounds on their negative moments.

Let $\{\vartheta_{t}\}_{t\geq 0}$ denote the family of shift mappings for
the LBM $\mathcal{B}$, which is defined by
$\vartheta_{t}(\omega'):=\theta_{F^{-1}_{t}(\omega')}(\omega')$
for $t\geq 0$ and $\omega'\in\Omega'$ and satisfies
$F^{-1}_{s+t}=F^{-1}_{s}+F^{-1}_{t}\circ\vartheta_{s}$ and hence
$\mathcal{B}_{t+s}=\mathcal{B}_{t}\circ\vartheta_{s}$ for $s,t\geq 0$
on $\Lambda^\om$ by virtue of $F_{t+s}=F_{t}+F_{s}\circ\theta_{t}$,
$s,t\geq 0$ (cf.\ \cite[Subsection~A.3.2]{CF12}).

\begin{prop} \label{prop:negmom_tau}
Let $q>0$. Then $\bbP$-a.s., for any $\kappa> 2+ \tilde \xi (q)$ and any $R\geq 1$
there exists a random constant $C_6=C_6(X,\gamma,R,q,\kappa)>0$ such that 
\begin{align} \label{eq:mom_exit}
 E_x\bigl[\tau_{B(x,r)}^{-q}\bigr] \leq C_6 r^{-\kappa},
	\qquad \forall  x\in B(R),\, r\in (0,1],
\end{align}
\end{prop}
\begin{proof}
Since \eqref{eq:mom_exit} is weaker for larger $\kappa$ and smaller $R$,
it suffices to show \eqref{eq:mom_exit} $\bbP$-a.s.\ for each $\kappa$ and $R$. 
First we note that, letting $n\to\infty$ in \cite[Proposition~2.12]{GRV13}
by using \cite[Lemma~2.8]{GRV13} (see also Theorem~\ref{thm:constrF} below)
and Fatou's lemma, we get
\begin{align} \label{eq:negmom_annealed}
 \bbE E_x \big[ \tau_{B(x,r)}^{-q}\big]\leq c r^{-\tilde \xi(q)},
	\qquad \forall x\in \bbR^2, \, r\in (0,1],
\end{align}
for some $c=c(\gamma, q)>0$. As in the proof of Lemma~\ref{lem:vol_dec} above
let $r_n:=2^{-n}R$ and $\Xi_{R,n}$ be defined as in \eqref{def:grid}
for any $n\geq 1$.
In the sequel we write $E_\mu$ for the expectation operator associated with
the law $P_\mu$ of a Brownian motion with initial distribution $\mu$. Let
$x\in \bbR^2$ and let $\mu_{x,r_n}:=P_x[\mathcal{B}_{\tau_{B(x,r_n)}} \in \cdot]$
be the distribution of the LBM upon exiting $B(x,r_n)$.
For any $z\in \partial B(x,r_n)$, since $B(z, r_{n})\subset B(x,2 r_{n})$
and hence $\tau_{B(z, r_{n})}\leq \tau_{B(x, 2 r_{n})}$,
by using \eqref{eq:negmom_annealed} we get
\begin{align*}
 \bbE E_z \bigl[ \tau_{B(x,2 r_{n})}^{-q}\bigr]
	\leq \bbE E_z \bigl[ \tau_{B(z, r_{n})}^{-q}\bigr]
	\leq c r_{n}^{-\tilde \xi(q)},
\end{align*}
provided $n$ is large enough so that $r_{n}\leq 1$.
By Fubini's theorem, the $\mu_{x,r_n}(dz)$-integral of this inequality becomes
\begin{align*}
 \bbE E_{\mu_{x,r_n}} \bigl[ \tau_{B(x,2 r_{n})}^{-q}\bigr]
	\leq  c r_{n}^{-\tilde \xi(q)}.
\end{align*}
Let now $\kappa> 2+\tilde \xi(q)$ and $n_0 \geq 1$ satisfying $r_{n_0}\in [\frac{1}{2},1]$
be fixed. Then for all $n\geq n_0$ we obtain by \v{C}eby\v{s}ev's inequality,
\begin{align*}
 \bbP\Bigl[\max_{x\in \Xi_{R,n+1}} E_{ \mu_{x,r_n}} \bigl[ \tau_{B(x,2 r_{n})}^{-q}\bigr]\geq  r_{n}^{-\kappa}\Bigr]
	&\leq  r_{n}^\kappa \sum_{x\in \Xi_{R,n+1}} \bbE E_{\mu_{x,r_n}} \bigl[ \tau_{B(x,2 r_{n})}^{-q}\bigr]
 \leq \frac c { 2^{n(\kappa-\tilde\xi(q)-2)}}
\end{align*}
for some $c=c(\gamma,R,q,\kappa)>0$. Hence by our choice of $\kappa$,
\begin{align*}
 \sum_{n\geq n_0} \bbP\Bigl[\max_{x\in \Xi_{R,n+1}}
	E_{\mu_{x,r_n}} \bigl[ \tau_{B(x,2r_{n})}^{-q}\bigr]\geq r_{n}^{-\kappa}\Bigr]
	< \infty
\end{align*}
and we apply the Borel-Cantelli lemma to obtain that $\bbP$-a.s.\ for all
$n\geq n_0$ and for all $x\in \Xi_{R,n+1}$,
\begin{align} \label{eq:neg_momhit}
 E_{\mu_{x,r_n}} \bigl[ \tau_{B(x,2r_{n})}^{-q}\bigr] \leq C r_{n}^{-\kappa}
\end{align}
for some random constant $C=C(X,\gamma, R,q,\kappa)>0$.

Now for any $r\in (0,1]$ we choose $n \geq n_0$ such that $r_n\leq \frac{2}{5}r < 2 r_{n}$.
For all $y\in B(R)$, by construction there exists $x\in \Xi_{R,n+1}$ such that
$|x-y|\leq \frac 1 2 r_{n}$.
Furthermore, from $B(x,r_n)\subset B(x,2 r_{n}) \subset B(y,r)$ we have
$\tau_{B(x,r_n)}\leq \tau_{B(x,2 r_{n})}
	=\tau_{B(x,r_n)}+\tau_{B(x,2 r_{n})}\circ\vartheta_{\tau_{B(x,r_n)}}
	\leq \tau_{B(y,r)}$,
and therefore by the strong Markov property \cite[Theorem~A.1.21]{CF12}
of $\mathcal{B}$,
\begin{align*}
 &E_y\bigl[ \tau_{B(y,r)}^{-q} \bigr]
	\leq E_y\bigl[ \tau_{B(x,2 r_{n})}^{-q} \bigr]
	\leq E_y\Bigl[\bigl( \tau_{B(x,r_n)}+ \tau_{B(x,2 r_{n})} \circ \vartheta_{\tau_{B(x,r_n)}} \bigr)^{-q} \Bigr] \\
 &\mspace{30mu}\leq E_y\Bigl[\bigl( \tau_{B(x,2 r_{n})} \circ \vartheta_{\tau_{B(x,r_n)}} \bigr)^{-q} \Bigr]
	= E_y\Bigl[E_{\mathcal{B}_{\tau_{B(x,r_n)}}} \bigl[\tau_{B(x,2 r_{n})}^{-q}\bigr] \Bigr]
	=E_{\mu^y_{x,r_n}}\bigl[\tau_{B(x,2 r_{n})}^{-q}\bigr],
\end{align*}
where $\mu^y_{x,r_n}:=P_y\bigl[\mathcal{B}_{\tau_{B(x,r_n)}} \in \cdot\bigr]$.
Since $\mu^y_{x,r_n}=P_y\bigl[B_{T_{B(x,r_n)}} \in \cdot\bigr]$ by
$\mathcal{B}_{\tau_{B(x,r_{n})}}=B_{T_{B(x,r_{n})}}$,
the exact formula for the distribution of a Brownian motion upon exiting balls
(see e.g.\ \cite[Exercise~4.2.24]{KS91})
implies that $\mu^y_{x,r_n}\leq c \mu_{x,r_n}$ for some explicit constant $c>0$
(this can be regarded as an application of the scale-invariant elliptic Harnack
inequality). Thus
$E_y\bigl[ \tau_{B(y,r)}^{-q} \bigr]
	\leq E_{\mu^y_{x,r_n}}\bigl[\tau_{B(x,2 r_{n})}^{-q}\bigr]
	\leq c E_{\mu_{x,r_n}}\bigl[\tau_{B(x,2 r_{n})}^{-q}\bigr]$
and the claim follows from \eqref{eq:neg_momhit}.
\end{proof}

\begin{prop} \label{prop:tailtau1}
$\bbP$-a.s., for any  $\beta> \alpha_1$ and any $R\geq 1$ there exist
random constants $C_i=C_i(X,\gamma,R,\beta)>0$, $i=7,8$, such that 
\begin{align*}
 P_x[\tau_{B(x,r)}\leq t]
	\leq C_7 \exp\Bigl(-C_8 (r^\beta /t )^{\frac{1}{\beta -1}}\Bigr),
	\qquad \forall  x\in B(R), \, r\in (0,1], \, t>0.
\end{align*}
\end{prop}
\begin{proof}
Let $\beta>\alpha_1$ and $R\geq 1$. By \cite[Theorem~7.2]{GK14}
it is enough to show that there exist $\varepsilon\in (0, 1)$ and $\delta>0$
such that $P_x[\tau_{B(x,r)}\leq \delta r^\beta]\leq \varepsilon$
for all $x\in B(R)$ and $r\in (0,1]$.
Indeed, let $\varepsilon\in (0, 1)$ and set $q:=2/\gamma$,
$\kappa:=\frac{1}{2}(\alpha_1+\beta)q$ and $\delta:=(\varepsilon/C_6)^{1/q}$, so that
$\kappa>\alpha_1 q=2+\tilde \xi(q)$ by $\beta> \alpha_1= (2+ \tilde \xi (q))/q$.
Then for any $x\in B(R)$ and $r\in (0,1]$,
by \v{C}eby\v{s}ev's inequality and Proposition~\ref{prop:negmom_tau} we have
\begin{align*}
 P_x[\tau_{B(x,r)}\leq \delta r^\beta]
	=P_x\bigl[(\tau_{B(x,r)})^{-q}\geq (\delta r^\beta)^{-q}\bigr]
	\leq C_6 \delta^q \, r^{\beta q-\kappa} \leq \varepsilon,
\end{align*}
proving the claim.
\end{proof}

\section{Strong Feller property of the resolvents} \label{sec:feller}
In this section we prove that the resolvent operator of the killed LBM
$\mathcal{B}^U$ has the strong Feller property. We will mainly follow
the arguments in \cite[Theorem~2.4]{GRV13a}, where the strong Feller
property of the original LMB $\mathcal{B}$ is established.
The essential ingredients are a coupling lemma and the following lemma.

\begin{lemma}[{\cite[Lemma~2.19]{GRV13v2}}] \label{lem:Fn_Tsmall}
$\bbP$-a.s., for all $R>0$,
\begin{align*}
 \lim_{t\downarrow 0} \sup_{n\geq 1} \sup_{x\in B(R)}E_x[F^n_t]=0.
\end{align*}
\end{lemma}

\begin{remark}
The article \cite{GRV13v2} is an earlier version of \cite{GRV13},
but Lemma~\ref{lem:Fn_Tsmall} has been removed from the latter,
which is why we still cite \cite{GRV13v2} in this paper.
\end{remark}

\begin{prop} \label{prop:strFeller}
$\bbP$-a.s., for any non-empty open set $U\subset \bbR^2$ and for any
$\lambda>0$ the resolvent operator $R_\lambda^U$ is strong Feller, i.e.\ it maps
Borel measurable bounded functions on $U$ to continuous bounded functions on $U$. 
\end{prop}
\begin{proof}
Throughout this proof, we fix any environment $\om \in \Omega$ such that all the
conclusions of Proposition~\ref{prop:pcaf} i), iv) and Lemma~\ref{lem:Fn_Tsmall} hold.
Note that by Proposition~\ref{prop:pcaf}, Lemma~\ref{lem:Fn_Tsmall} and
Fatou's lemma we have
\begin{align}\label{eq:F_Tsmall}
 \lim_{t\downarrow 0} \sup_{x\in B(R)}E_x[F_t]=0, \qquad \forall R\geq 1.
\end{align}

Let $U$ be a non-empty open subset of $\bbR^2$, let $\lambda>0$ and
let $f: U\rightarrow \bbR$ be Borel measurable and bounded.
Recall that $T_U=\inf\{s\geq 0:\, B_s\not\in U\}$ denotes the exit time of
the Brownian motion $B$ from $U$. Since $\tau_U=F_{T_U}$,
$R_{\lambda}^{U}f$ can be written as
\begin{align} \label{eq:decompRes}
 R_\lambda^Uf(x)&=E_x\Bigl[ \int_0^{\tau_U} e^{-\lambda t} f(\mathcal{B}_t) \, dt \Bigr]
	=E_x\Bigl[ \int_0^{T_U} e^{-\lambda F_t} f(B_t) \, dF_t \Bigr] \nonumber \\
 &= E_x\Bigl[ \int_0^{T_U \wedge \varepsilon } e^{-\lambda F_t} f(B_t) \, dF_t \Bigr]
	+ E_x\Bigl[ \int_{T_U\wedge \varepsilon}^{T_U} e^{-\lambda F_t} f(B_t) \, dF_t \Bigr] \nonumber \\
 &=:N_\varepsilon(x)+R_\lambda^{U,\varepsilon}f(x)
\end{align}
for any $x\in\bbR^2$ and any $\varepsilon >0$. It is immediate that
\begin{align} \label{eq:decompRes_Nepsilon}
 |N_\varepsilon(x)| \leq \|f\|_\infty \, E_x[F_\varepsilon],
\end{align}
whereas for $R_\lambda^{U,\varepsilon}f(x)$ the Markov property
of $B$ gives
\begin{align}\label{eq:decompRes_RUepsilon1}
 &R_\lambda^{U,\varepsilon}f(x)
	=E_x\Bigl[ \indicator_{\{T_U>\varepsilon\}} \int_{ \varepsilon}^{T_U} e^{-\lambda F_t} f(B_t) \, dF_t \Bigr] \\
 &\mspace{30mu}=E_x\biggl[ \indicator_{\{T_U>\varepsilon\}} e^{-\lambda F_\varepsilon} E_{B_\varepsilon} \Bigl[\int_{0}^{T_U} e^{-\lambda F_t} f(B_t) \, dF_t \Bigr] \biggr]
 =E_x\Bigl[ \indicator_{\{T_U>\varepsilon\}} e^{-\lambda F_\varepsilon} R_\lambda^U f(B_\varepsilon)\Bigr].
 \nonumber
\end{align}

To estimate $R_\lambda^{U,\varepsilon}f(x)-R_\lambda^{U,\varepsilon}f(y)$
we use the coupling lemma \cite[Lemma~2.9]{GRV13}, which allows to construct
for any $x,y \in \bbR^2$ a couple $(B^x,B^y)$ of Brownian motions
$B^x=(B^x_t)_{t\geq 0}$ and $B^y=(B^y_t)_{t\geq 0}$ with $(B^x_0,B^y_0)=(x,y)$
such that $B^x_t=B^y_t$ for any $t\in[T_{xy},\infty)$ for a random time
$T_{xy}$ satisfying
\begin{align}\label{eq:coupling}
 \lim_{\delta\downarrow 0}\sup_{x,y\in \bbR^2,\,|x-y|\leq\delta}P_{x,y}[T_{xy}\geq\varepsilon]=0
\end{align}
for any $\varepsilon>0$, where $P_{x,y}$ denotes the law of $(B^x,B^y)$.
Let $E_{x,y}$ denote the expectation under $P_{x,y}$ and set
$T^x_U:=T_U(B^x)$, $T^y_U:=T_U(B^y)$, $F^{x}_{t}:=F_{t}(B^x)$ and
$F^{y}_{t}:=F_{t}(B^y)$, with $T_U$ and $F_t$ for $t\geq 0$ regarded as
functions on the path space $\Omega'=C([0,\infty),\bbR^2)$.
Then according to \cite[Proof of Theorem~2.4]{GRV13a}, for any $\varepsilon>0$,
\begin{align}\label{eq:strFellerDiffF}
 \lim_{\delta\downarrow 0}\sup_{x,y\in B(R),\,|x-y|\leq\delta}
	E_{x,y}\Bigl[ \bigl| e^{-\lambda F^x_\varepsilon}- e^{-\lambda F^y_\varepsilon}\bigr|\Bigr]=0,
	\qquad \forall R\geq 1,
\end{align}
whose proof we repeat here for the sake of completeness.
Indeed, for any $\varepsilon'\in(0,\varepsilon]$, since
$F^{x}_{\varepsilon}-F^{x}_{T_{xy}}=F^{y}_{\varepsilon}-F^{y}_{T_{xy}}>0$
$P_{x,y}$-a.s.\ on $\{T_{xy}<\varepsilon\}$ by Proposition~\ref{prop:pcaf} i), ii),
\begin{align*}
 &E_{x,y}\Bigl[ \bigl| e^{-\lambda F^x_\varepsilon}- e^{-\lambda F^y_\varepsilon}\bigr|\Bigr]
	\leq 2P_{x,y}[T_{xy}\geq\varepsilon]
	+E_{x,y}\Bigl[ \indicator_{\{T_{xy}<\varepsilon\}} \bigl| e^{-\lambda F_\varepsilon^x}- e^{-\lambda F_\varepsilon^y}\bigr| \Bigr] \\
 &\mspace{30mu}=2P_{x,y}[T_{xy}\geq\varepsilon]
	+E_{x,y}\Bigl[ \indicator_{\{T_{xy}<\varepsilon\}}e^{-\lambda(F^x_\varepsilon-F^x_{T_{xy}})}\Bigl| e^{-\lambda F^x_{T_{xy}}}-e^{-\lambda F^y_{T_{xy}}}\Bigr|\Bigr]\\
 &\mspace{30mu}\leq 2P_{x,y}[T_{xy}\geq\varepsilon]+E_{x,y}\Bigl[ \bigl| \lambda F^x_{T_{xy}}-\lambda F^y_{T_{xy}}\bigr|\wedge 1 \Bigr] \\
 &\mspace{30mu}\leq 2P_{x,y}[T_{xy}\geq\varepsilon]+P_{x,y}[T_{xy}\geq\varepsilon']
	+\lambda E_{x,y}\bigl[\indicator_{\{T_{xy}<\varepsilon'\}}(F^x_{\varepsilon'}+F^y_{\varepsilon'})\bigr]\\
 &\mspace{30mu}\leq 3P_{x,y}[T_{xy}\geq\varepsilon']+\lambda\bigl(E_{x}[F_{\varepsilon'}]+E_{y}[F_{\varepsilon'}]\bigr)
\end{align*}
and taking $\lim_{\varepsilon'\downarrow 0}\limsup_{\delta\downarrow 0}\sup_{x,y\in B(R),\,|x-y|\leq\delta}$
yields \eqref{eq:strFellerDiffF} by \eqref{eq:coupling} and \eqref{eq:F_Tsmall}.

Now let $x,y\in\bbR^2$ and $\varepsilon>0$. From \eqref{eq:decompRes_RUepsilon1} we obtain
\begin{align}\label{eq:decompRes_RUepsilon2}
 \bigl| R_\lambda^{U,\varepsilon}f(x)- R_\lambda^{U,\varepsilon}f(y)\bigr|
	&= \Bigl| E_{x,y} \Bigl[ \indicator_{\{T^x_U>\varepsilon\}}  e^{-\lambda F^x_\varepsilon} R_\lambda^U f(B^x_\varepsilon)
	- \indicator_{\{T^y_U>\varepsilon\}}  e^{-\lambda F^y_\varepsilon} R_\lambda^U f(B^y_\varepsilon)\Bigr] \Bigr| \nonumber \\
&\leq \Bigl| E_{x,y} \Bigl[ \indicator_{\{T^x_U>\varepsilon\}}  e^{-\lambda F^x_\varepsilon} \bigl( R_\lambda^U f(B^x_\varepsilon) - R_\lambda^U f(B^y_\varepsilon) \bigr) \Bigr] \Bigr| \nonumber \\
&\mspace{50mu}+\Bigl| E_{x,y} \Bigl[ \bigl( \indicator_{\{T^x_U>\varepsilon\}}  e^{-\lambda F^x_\varepsilon}-\indicator_{\{T^y_U>\varepsilon\}}  e^{-\lambda F^y_\varepsilon}\bigr)  R_\lambda^U f(B^y_\varepsilon)\Bigr] \Bigr|.
\end{align}
Since on the event $\{T_{xy}<\varepsilon\}$ we have
$B^x_\varepsilon=B^y_\varepsilon$ and hence
$R_\lambda^U f(B^x_\varepsilon)=R_\lambda^U f(B^y_\varepsilon)$,
the first term in \eqref{eq:decompRes_RUepsilon2} can be estimated from above by
\begin{multline}\label{eq:decompRes_RUepsilon3}
 E_{x,y} \Bigl[ \bigl| R_\lambda^U f(B^x_\varepsilon) - R_\lambda^U f(B^y_\varepsilon) \bigr| \Bigr]
	=E_{x,y} \Bigl[ \indicator_{\{T_{xy}\geq\varepsilon\}} \bigl| R_\lambda^U f(B^x_\varepsilon) - R_\lambda^U f(B^y_\varepsilon) \bigr| \Bigr]\\
 \leq 2 \| R_\lambda^U f\|_{\infty} P_{x,y}[T_{xy}\geq\varepsilon]
	\leq 2 \lambda^{-1} \|f\|_{\infty} P_{x,y}[T_{xy}\geq\varepsilon],
\end{multline}
where we used the trivial bounds
$0\leq\indicator_{\{T^x_U>\varepsilon\}} e^{-\lambda F^x_\varepsilon}\leq 1$ and
$\| R_\lambda^U f \|_\infty\leq \lambda^{-1} \|f\|_\infty$.
On the other hand, the second term in \eqref{eq:decompRes_RUepsilon2}
is less than or equal to
\begin{align}
 &\lambda^{-1} \|f\|_\infty \, E_{x,y} \Bigl[ \bigl|\indicator_{\{T^x_U>\varepsilon\}}-  \indicator_{\{T^y_U>\varepsilon\}}\bigr| e^{-\lambda F_\varepsilon^x}+ \indicator_{\{T^y_U>\varepsilon\}} \bigl|  e^{-\lambda F_\varepsilon^x}- e^{-\lambda F_\varepsilon^y}\bigr| \Bigr] \nonumber\\
 &\mspace{30mu}\leq\lambda^{-1} \|f\|_\infty \, E_{x,y} \Bigl[ \bigl|\bigl(1-\indicator_{\{T^x_U\leq\varepsilon\}}\bigr)-\bigl(1-\indicator_{\{T^y_U\leq\varepsilon\}}\bigr)\bigr| + \bigl|  e^{-\lambda F_\varepsilon^x}- e^{-\lambda F_\varepsilon^y}\bigr| \Bigr] \nonumber\\
 &\mspace{30mu}\leq\lambda^{-1} \|f\|_\infty\Bigl( P_{x,y}[T^x_U\leq\varepsilon]+P_{x,y}[T^y_U\leq\varepsilon]+E_{x,y}\Bigr[ \bigl| e^{-\lambda F_\varepsilon^x}- e^{-\lambda F_\varepsilon^y}\bigr|\Bigr] \Bigr).
 \label{eq:decompRes_RUepsilon4}
\end{align}
Noting $P_{x,y}[T^x_U\leq\varepsilon]=P_{x}[T_U\leq\varepsilon]$
and $P_{x,y}[T^y_U\leq\varepsilon]=P_{y}[T_U\leq\varepsilon]$,
from \eqref{eq:decompRes}, \eqref{eq:decompRes_Nepsilon},
\eqref{eq:decompRes_RUepsilon2}, \eqref{eq:decompRes_RUepsilon3} and
\eqref{eq:decompRes_RUepsilon4} we get
\begin{align}
 &\bigl|R_\lambda^U f(x)-R_\lambda^U f(y)\bigr| \nonumber\\
 &\mspace{30mu}\leq\|f\|_{\infty}\bigl(E_x[F_{\varepsilon}]+E_y[F_{\varepsilon}]\bigr)
	+\lambda^{-1}\|f\|_{\infty}\bigl(P_{x}[T_U\leq\varepsilon]+P_{y}[T_U\leq\varepsilon]\bigr) \nonumber \\
 &\mspace{80mu}+\lambda^{-1}\|f\|_{\infty}\Bigl(2P_{x,y}[T_{xy}\geq\varepsilon]+E_{x,y}\Bigr[ \bigl| e^{-\lambda F_\varepsilon^x}- e^{-\lambda F_\varepsilon^y}\bigr|\Bigr]\Bigr).
 \label{eq:strFellerDifference}
\end{align}

Finally, let $x\in U$ and choose $r_x>0$ so that $B(x,2r_x)\subset U$.
Then for any $y\in B(x,r_x)$, $T_{B(y,r_x)}\leq T_U$
by $B(y,r_x)\subset B(x,2r_x)\subset U$ and hence
\begin{align}\label{eq:strFellerExitU}
 P_{y}[T_U\leq\varepsilon]\leq P_{y}[T_{B(y,r_x)}\leq\varepsilon]
	\leq 2\exp\bigl(-r_x^2/(4\varepsilon)\bigr)
\end{align}
(see e.g.\ \cite[Proposition~2.6.19]{KS91} for the latter inequality). Now we can
easily conclude $\limsup_{y\to x}\bigl|R_\lambda^U f(x)-R_\lambda^U f(y)\bigr|=0$
by taking the supremum in $y\in B(x,r_x)$ of the second line of
\eqref{eq:strFellerDifference}, using \eqref{eq:coupling} and
\eqref{eq:strFellerDiffF} to let $y\to x$ and then using \eqref{eq:F_Tsmall}
and \eqref{eq:strFellerExitU} to let $\varepsilon\downarrow 0$.
Thus $R_\lambda^U f$ is continuous on $U$.
\end{proof}

\section{Continuity and upper bounds of the heat kernels} \label{sec:cont_ub}
Throughout Sections~\ref{sec:cont_ub} and \ref{sec:lb} we fix any environment
$\om \in \Omega$ such that all the conclusions of
Proposition~\ref{prop:pcaf} i), iv), Lemma~\ref{lem:vol_dec},
Propositions~\ref{prop:tailtau1}, \ref{prop:strFeller} and \ref{prop:corrF_M} hold.

The purpose of this section is to prove Theorem~\ref{thm:dir_hk} below on
the continuity of the heat kernels as well as Theorem~\ref{thm:on-diag_glob}.
Recall that $\mathcal{F}$ equipped with the norm
$\|f\|^2_{\mathcal{F}}:=\mathcal{E}(f,f)+\|f\|^2_{L^2(\bbR^2,M)}$
is a Hilbert space. For any open set $U \subset \bbR^2$, we define
$\mathcal{F}_U$ to be the closure in $(\mathcal{F}, \|\cdot\|_{\mathcal{F}})$
of the set of all functions in $\mathcal{F}$
whose $M$-essential supports in $\bbR^2$ are compact subsets of $U$.
It is well known that $(\mathcal{E}, \mathcal{F}_U)$ is the Dirichlet form
associated with the killed Liouville Brownian motion $\mathcal{B}^U$ and that
it is regular on $L^2(U,M)$ (see e.g.\ \cite[Theorems~4.4.2 and 4.4.3]{FOT11}).
The associated non-positive self-adjoint operator on $L^2(U,M)$ is denoted
by $\mathcal{L}_U$, its domain by $\mathcal{D}(\mathcal{L}_U)$, and
the associated semigroup and resolvent operators by $(T_t^U)_{t>0}$
and $(G^U_{\lambda})_{\lambda >0}$, respectively.

\begin{theorem} \label{thm:dir_hk}
For any non-empty open set $U\subset \bbR^2$ the following hold:

\begin{enumerate}
 \item[i)] There exists a (unique) jointly continuous function
	$p^U=p^U_t(x,y):(0,\infty)\times U \times U \rightarrow [0,\infty)$
	such that for all $(t,x)\in (0,\infty) \times U$,
	$P_x[\mathcal{B}^U_t \in dy]=p^U_t(x,y) \, M(dy)$,
	which we refer to as the \emph{Dirichlet Liouville heat kernel} on $U$.
\item[ii)] The semigroup operator $P_t^U$ is strong Feller, i.e.\ it maps Borel
	measurable bounded functions on $U$ to continuous bounded functions on $U$.
\item[iii)] If $U$ is connected, then $p^U_t(x,y)\in (0,\infty)$ for any
	$(t,x,y)\in(0,\infty)\times U \times U$, and in particular the Dirichlet
	form $(\mathcal{E}, \mathcal{F}_U)$ of $\mathcal{B}^U$ is irreducible.
\end{enumerate}
\end{theorem}

See \cite[Section~1.6, p.\ 55]{FOT11} for the definition of the irreducibility of
a symmetric Dirichlet form and \cite[Theorem~4.7.1 (i) and Exercise~4.7.1]{FOT11}
for its probabilistic consequences.

From now on we will write $p_t(\cdot,\cdot)$ instead of
$p^{\bbR^2}_t(\cdot,\cdot)$ and call it the (global)
\emph{Liouville heat kernel}. Note that Theorem~\ref{thm:cont_hk}
follows directly from Theorem~\ref{thm:dir_hk} by choosing $U=\bbR^2$.

\subsection{The heat kernel on bounded open sets} \label{sec:Ubounded}

In this subsection we will prove Theorem~\ref{thm:dir_hk} for a fixed non-empty
bounded open set $U\subset\bbR^2$. The case of unbounded open sets will be
treated later in Subsection~\ref{sec:Uunbounded}. We denote by $\|f\|_p$
the $L^p(U,M)$-norm for $p\geq 1$ and by $\langle \cdot, \cdot\rangle$ the
$L^2(U,M)$-inner product. Let $R\geq 1$ be such that $U\subset B(R)$.

%

\begin{prop}[Faber-Krahn-type inequality] \label{prop:FK}
The spectrum of $-\mathcal{L}_U$ is discrete, and for its smallest eigenvalue
$\lambda_1(U)$ there exists $C_9=C_9(X,\gamma,R)>0$ such that
\begin{align} \label{eq:FK}
 \lambda_{1}(U) \geq  \frac{C_9}{M(U) \log\bigl(2+\frac{1}{M(U)}\bigr)}.
\end{align}
\end{prop}
\begin{proof}
First, it is elementary to verify that $\sup_{x\in U}\|g_U(x,\cdot)\|_2 <\infty$
by \eqref{eq:killed_green} and a calculation similar to \eqref{eq:Mnlog} based
on Lemma~\ref{lem:vol_dec}, so that $g_U \in L^2(U\times U, M\times M)$,
$G_U f(x)=\bigl\langle g_U(x,\cdot),f\bigr\rangle\in\bbR$ for $x\in U$ for any
$f\in L^2(U,M)$, and $G_U$ defines a bounded linear operator on $L^{2}(U,M)$
which is Hilbert-Schmidt and hence (see e.g.\ \cite[Theorem~4.2.16]{Da07})
compact. Then in view of \cite[(1.5.3) and Theorem~4.2.3 (ii)]{FOT11}
the Dirichlet form $(\mathcal{E}, \mathcal{F}_U)$ of $\mathcal{B}^U$ is transient
in the sense of \cite[(1.5.4)]{FOT11}, or equivalently in the sense of
\cite[(1.5.6)]{FOT11} by \cite[Theorem~1.5.1]{FOT11}, which implies that
$\bigl\{u\in\mathcal{F}_U:\,\mathcal{E}(u,u)=0\bigr\}=\{0\}$,
namely $\mathcal{L}_U$ is injective. Now by
\cite[Theorem~4.2.6, Theorem~1.5.4 (i) and Theorem~1.5.2 (iii)]{FOT11},
$G_U f\in\mathcal{D}(\mathcal{L}_U)$ and $-\mathcal{L}_U G_U f=f$
for any $f\in L^2(U,M)$, which together with
the injectivity of $\mathcal{L}_U$ yields $(-\mathcal{L}_U)^{-1}=G_U$.
In particular, $(-\mathcal{L}_U)^{-1}$ is compact, and therefore the
spectrum of $-\mathcal{L}_U$ is discrete by \cite[Corollary~4.2.3]{Da95}.

For the proof of \eqref{eq:FK}, note that by the spectral decomposition of
the compact self-adjoint operator $(-\mathcal{L}_U)^{-1}=G_U$
(see e.g.\ \cite[Theorem~4.2.2]{Da95}) and $g_U \geq 0$,
\begin{align} \label{eq:lambda1_var}
 \lambda_1(U)^{-1}= \sup \bigl\{ \langle G_Uf , f \rangle:
	\, f\in L^{2}(U,M), \, f\geq 0, \, \|f \|_2=1 \bigr\}.
\end{align}
Let $f\in L^{2}(U,M)$ satisfy $f\geq 0$ and $\|f \|_2=1$.
Setting $\nu:=\pi \alpha_2/2=\frac{\pi}{4}(2-\gamma)^2$ and noting that
$g_U \leq g_{B(R+1)}$ by $U\subset B(R)\subset B(R+1)$, we have
\begin{align} \label{eq:decompGU}
 \langle G_U f , f \rangle\leq  \langle G_{B(R+1)}f , f \rangle
	&\leq \int_U \int_U \exp\bigl(\nu g_{B(R+1)}(x,y)\bigr) \, M(dy) \, M(dx) \\
 &\mspace{39mu}+\int_U \int_U  \frac{f(x) f(y)}{\nu} \log\Bigl(1+\frac{f(x)f(y)}{\nu} \Bigr) \, M(dy) \, M(dx), \nonumber
\end{align}
where we used the elementary inequality $ ab\leq a \log(1+a)+e^b$, valid for any
$a,b\in[0,\infty]$, with $a=\frac{f(x)f(y)}{\nu}$ and $b=\nu g_{B(R+1)}(x,y)$.
For the first integral in \eqref{eq:decompGU}, we have
\begin{align*}
 \int_U \int_U \exp\bigl(\nu g_{B(R+1)}(x,y)\bigr) \, M(dy) \, M(dx)
	\leq \int_U \int_U \frac{c}{ |x-y|^{\nu/\pi}} \, M(dy) \, M(dx)
\end{align*}
with $c=c(\gamma,R)>0$ by \eqref{eq:killed_green} and $U\subset B(R)$, and then
using Lemma~\ref{lem:vol_dec} with
$\varepsilon=\alpha_{2}/4\in(0,\alpha_2-\nu/\pi)$, we further obtain
\begin{align*}
 &\int_U \int_U \frac{1}{|x-y|^{\nu/\pi}} \, M(dy) \, M(dx)
	\leq \int_U \int_{B(x,2R)} \frac{1}{|x-y|^{\nu/\pi}} \, M(dy) \, M(dx) \\
 &\mspace{120mu}\leq \sum_{n=0}^\infty \int_U \int_{B(x,2^{1-n}R) \setminus B(x,2^{-n}R)} (2^{-n}R)^{-\nu/\pi} \, M(dy) \, M(dx)
	\leq C M(U)
\end{align*}
for some $C=C(X,\gamma,R)>0$.
On the other hand, setting $\bar M_U:=M(\cdot \cap U)/M(U)$,
we can write the second term in \eqref{eq:decompGU} as
\begin{align*}
 M(U)^2 \int_U \int_U  \frac{f(x) f(y)}{\nu} \log\Bigl(1+\frac{f(x) f(y)}{\nu} \Bigr)
	\, \bar M_U(dy) \, \bar M_U(dx).
\end{align*}
For the homeomorphisms $H,I:[0,\infty)\to[0,\infty)$ defined by
$H(s):=s^2$ and $I(s):=s \log(1+s)$, we easily see that the function
$H\circ I^{-1}$ convex, and we apply Jensen's inequality to get
\begin{align*}
 &H\circ I^{-1} \biggl( \int_U \int_U \frac{f(x) f(y)}{\nu} \log\Bigl(1+\frac{f(x) f(y)}{\nu} \Bigr) \, \bar M_U(dy) \, \bar M_U(dx)\biggr) \\
 &\mspace{200mu}\leq \int_U \int_U \Bigl(\frac{f(x) f(y)}{\nu} \Bigr)^2 \, \bar M_U(dy) \, \bar M_U(dx)
	= \frac{1}{\nu^2 M(U)^2},
\end{align*}
where we used $ \|f \|_2=1$. Hence
\begin{align*}
 &\int_U \int_U  \frac{f(x) f(y)}{\nu} \log\Bigl(1+\frac{f(x) f(y)}{\nu} \Bigr) \, \bar M_U(dy) \, \bar M_U(dx) \\
 &\mspace{200mu}\leq I\circ H^{-1} \Bigl(\frac{1}{\nu^2 M(U)^2}\Bigr)
	=\frac{1}{\nu M(U)} \log\Bigl( 1+ \frac{1}{\nu M(U)}\Bigr).
\end{align*}
Finally, we combine the above considerations to conclude that
\begin{align*}
 \langle G_Uf , f \rangle\leq C M(U)+ \frac{1}{\nu} M(U) \log\Bigl( 1+ \frac{1}{\nu M(U)}\Bigr) 
	\leq C M(U) \log\Bigl( 2+ \frac{1}{M(U)}\Bigr)
\end{align*}
for some constants $C$ large enough,
which together with \eqref{eq:lambda1_var} yields the claim.
\end{proof}

In the next proposition we derive from the above Faber-Krahn inequality a
Nash-type inequality and thereby an on-diagonal estimate on $(T^U_t)_{t>0}$
of the same form as stated for $p_t=p_t(x,y)$ in Theorem~\ref{thm:on-diag_glob}.
In particular, $(T^U_t)_{t>0}$ turns out to be ultracontractive,
i.e.\ $T^U_t\bigl(L^2(U,M)\bigr) \subset L^\infty(U,M)$ and
$T^U_t: L^2(U,M) \to L^\infty(U,M)$ is a bounded linear operator for all $t>0$.
Recall that for each $t>0$, $T^U_t$ is a self-adjoint Markovian operator on
$L^2(U,M)$ and hence canonically extends to a bounded linear operator on
$L^1(U,M)$ with operator norm at most $1$ (see e.g.\ \cite[(1.5.2)]{FOT11}).
For a bounded linear operator $A:L^{1}(U,M)\to L^{\infty}(U,M)$,
its operator norm will be denoted by $\| A \|_{L^1(U) \to L^\infty(U)}$.

\begin{prop} \label{prop:on_diag_up}
There exists a constant $C_{10}=C_{10}(X,\gamma, R)>0$ such that
\begin{align} \label{eq:ondiag_dirhk}
 \bigl\| T_t^{U} \bigr\|_{L^1(U) \to L^\infty(U)} \leq C_{10} t^{-1} \log(t^{-1}),
	\qquad \forall t\in (0,\tfrac 1 2].
\end{align}
\end{prop}
\begin{proof}
Since
$\bigl\| T_t^{U} \bigr\|_{L^1(U) \to L^\infty(U)} \leq \bigl\| T_t^{B(R)} \bigr\|_{L^1(B(R)) \to L^\infty(B(R))}$
by $U\subset B(R)$, it is enough to show \eqref{eq:ondiag_dirhk} for $\bigl(T_t^{B(R)}\bigr)_{t>0}$.
Recall that for any non-empty open subset $V$ of $B(R)$, the smallest eigenvalue
$\lambda_1(V)$ of $-\mathcal{L}_{V}$ admits the variational expression
\begin{align*} 
 \lambda_1(V) = \inf \biggl\{ \frac{ \mathcal{E}(f,f)}{\|f\|_2^2}:
	\, f \in \mathcal{F}_V, f \not= 0 \biggr\}
\end{align*}
(see e.g.\ \cite[Theorems~4.5.1 and 4.5.3]{Da95}), so that we can rewrite
the Faber-Krahn inequality of Proposition~\ref{prop:FK} for $V$ as
\begin{align} \label{eq:fk_new}
\|f\|_2^2 \leq C_{9}^{-1} \psi\bigl(M(V)\bigr) \mathcal{E}(f,f), \qquad \forall f\in \mathcal{F}_V,
\end{align}
where $\psi(s):=s \log(2+s^{-1})$ ($\psi(0):=0$).
Next we will verify that
\begin{align} \label{eq:fk_supp}
 \|f\|_2^2 \leq C_{9}^{-1} \psi\bigl(M(\supp[f])\bigr) \mathcal{E}(f,f),
	\qquad \forall f\in \mathcal{F}_{B(R)},
\end{align}
where $\supp[f]:=\supp_{B(R)}[f]$ denotes the $M$-essential support of $f$ in $B(R)$.
First, for $f\in \mathcal{F}_{B(R)}$ with $\supp[f]$ compact, \eqref{eq:fk_supp}
follows by choosing a decreasing sequence $(V_n)_{n\geq 1}$ of open subsets of
$B(R)$ with $\bigcap_{n\geq 1} V_n = \supp[f]$, applying \eqref{eq:fk_new} with
$V=V_n$ and letting $n\to \infty$. Next, for general $f\in \mathcal{F}_{B(R)}$,
as $|f|\in \mathcal{F}_{B(R)}$ and $\mathcal{E}(|f|,|f|)\leq \mathcal{E}(f,f)$
we may assume $f\geq 0$. Let $(f_n)_{n\geq 1} \subset\mathcal{F}_{B(R)}$ be a
sequence with $\supp[f_n]$ compact and $\lim_{n\to\infty}\|f_n-f\|_{\mathcal{F}}=0$,
where by $f\geq 0$ and \cite[Theorem~1.4.2 (v)]{FOT11} we may assume that
$f_n \geq 0$ for all $n$. Then since $f\wedge f_n \in  \mathcal{F}_{B(R)}$,
$\supp[f\wedge f_n]$ is a compact subset of $\supp[f]$ and
\begin{align*}
 \| f- f\wedge f_n \|_{\mathcal{F}}=\| (f-f_n)^+ \|_{\mathcal{F}}
	\leq \|f-f_n \|_{\mathcal{F}}\xrightarrow{n\to\infty} 0,
\end{align*}
we conclude \eqref{eq:fk_supp} for all $f\in \mathcal{F}_{B(R)}$
by letting $n\to\infty$ in \eqref{eq:fk_supp} for $f\wedge f_n$. 

Now, since $\psi:[0,\infty)\to[0,\infty)$ is strictly increasing,
\cite[Proposition~10.3]{BCLS95} and \eqref{eq:fk_supp} together imply that
\begin{align*}
 \|f\|_2^2 \leq 8C_{9}^{-1} \psi\bigl(4/\|f\|^2_2\bigr) \mathcal{E}(f,f)
	\qquad \text{for all $f\in \mathcal{F}_{B(R)}$ with $0< \|f\|_1\leq 1$.}
\end{align*}
In particular, for such $f$ we have $\theta(\|f\|_2^2)\leq \mathcal{E}(f,f)$
with $\theta(s):=\frac{1}{32}C_{9}s^2 /\log(2+s/4)$, and then by
\cite[Proposition~II.1]{Co96} we obtain
\begin{align} \label{eq:ondiag_coulhon}
 \bigl\| T_t^{B(R)} \bigr\|_{L^1(B(R)) \to L^\infty(B(R))} \leq m(t), \qquad \forall t> 0,
\end{align}
for the unique differentiable function $m:(0,\infty)\to(0,\infty)$ satisfying
\begin{align} \label{eq:ode_m}
m'(t)=-\theta\bigl(m(t)\bigr), \qquad \lim_{t \downarrow 0}m(t)=\infty.
\end{align}
It is immediate that $m=\Phi^{-1}$, where $\Phi:(0,\infty)\to(0,\infty)$ is a
decreasing diffeomorphism defined by $\Phi(s):=\int_s^\infty \theta(u)^{-1}\,du$,
and furthermore for all $s\in (0,\infty)$,
\begin{align*}
 \Phi(s)=\int_s^\infty \frac{32C_{9}^{-1}}{u^2} \log(2+u/4)\,du
	\leq \frac{80C_{9}^{-1}}{s}\log(2+s/4)=:\Psi(s),
\end{align*}
which means that $\Psi^{-1}(t)\geq \Phi^{-1}(t)$ for all $t\in (0, \infty)$
since $\Psi:(0,\infty)\to(0,\infty)$ is also a decreasing diffeomorphism.
Finally, for all $t\in(0,\frac{1}{2}]$ we easily see that
$\Psi\bigl(t^{-1} \log(t^{-1})\bigr) \leq Ct$ and hence that
\begin{align*}
 m(Ct)= \Phi^{-1}(Ct) \leq \Psi^{-1}(Ct)\leq t^{-1} \log(t^{-1}),
\end{align*}
and the claim then follows from \eqref{eq:ondiag_coulhon}.
\end{proof}

Now we prove Theorem~\ref{thm:dir_hk} for bounded open sets $U$. Given the
ultracontractivity of $(T^U_t)_{t>0}$ in Proposition~\ref{prop:on_diag_up}
and the strong Feller property in Proposition~\ref{prop:strFeller},
a general result in \cite{Da89} provides the existence of a continuous kernel
$p^U=p^U_t(x,y)$ for $(T^U_t)_{t>0}$, but we still have to identify
this kernel as the transition density of $\mathcal{B}^U$.

\begin{proof}[Proof of Theorem~\ref{thm:dir_hk} for bounded $U$]
We divide the proof of i) into several steps.

\emph{Step 1:} In the first step we show the existence of a jointly continuous
integral kernel $p^U=p^U_t(x,y)$ for $(T^U_t)_{t>0}$.
Being discrete by Proposition~\ref{prop:FK}, the spectrum of $-\mathcal{L}_U$ takes the
form of an unbounded non-decreasing sequence $(\lambda_n)_{n\geq 1}\subset[0,\infty)$
of eigenvalues repeated according to multiplicity, and there exists a complete
orthonormal system $(\varphi_n)_{n\geq 1}\subset \mathcal{D}(\mathcal{L}_U)$ of
$L^2(U,M)$ such that $-\mathcal{L}_U \varphi_n=\lambda_n \varphi_n$ for any
$n\geq 1$ (see e.g.\ \cite[Corollary~4.2.3]{Da95}).
Then $\varphi_{n}=e^{\lambda_{n}}T^{U}_{1}\varphi_{n}\in L^{\infty}(U,M)$
by Proposition~\ref{prop:on_diag_up}, so that we may choose a bounded Borel
measurable version of $\varphi_{n}$ for each $n$. Further, since
$R^U_{\lambda}\varphi_n$ is continuous on $U$ for any $\lambda>0$
by Proposition~\ref{prop:strFeller} and
\begin{align} \label{eq:res_phi}
R^U_\lambda \varphi_n=G^U_\lambda\varphi_n
	=(\lambda +\lambda_n)^{-1}\varphi_n \qquad \text{$M$-a.e.\ on $U$}
\end{align}
by \cite[Theorem 4.2.3 (ii)]{FOT11}, there exists a continuous version of
$\varphi_n$, which is unique, bounded, and still denoted by $\varphi_n$.
Then by \cite[Theorem~2.1.4]{Da89}, the series
\begin{align} \label{eq:pU_eig}
 p^U_t(x,y):=\sum_{n=1}^\infty e^{-\lambda_n t} \varphi_n(x) \varphi_n(y)
\end{align}
absolutely converges uniformly on $[\varepsilon, \infty) \times U \times U$
for any $\varepsilon>0$, from which the joint continuity of $p^U=p^U_t(x,y)$
follows, and \eqref{eq:pU_eig} defines an integral kernel for $(T^U_t)_{t>0}$,
namely for each $t>0$ and $f\in L^2(U,M)$,
\begin{align}\label{eq:pU_integral_kernel}
 T^U_t f(x) =\int_U p^U_t(x,y) f(y) \, M(dy) \qquad \text{for $M$-a.e.\ $x\in U$.}
\end{align}
Note that the boundedness of $\varphi_n$ together with the uniform
convergence of \eqref{eq:pU_eig} implies the boundedness of $p^U_t(x,y)$
on $[\varepsilon, \infty) \times U \times U$ for each $\varepsilon>0$,
and also that $p^U_t(x,y)\geq 0$ by a monotone class argument based on 
\eqref{eq:pU_integral_kernel} and the fact that
$T^U_t f\geq 0$ $M$-a.e.\ for any $f\in L^2(U,M)$ with $f\geq 0$.

\emph{Step 2:} In this step we show that $R^U_\lambda$ is absolutely continuous
with respect to the Liouville measure $M$ for any $\lambda>0$.
Let $A$ be a Borel subset of $U$ with $M(A)=0$. Then
$R^U_\lambda \indicator_A$ is continuous on $U$ by Proposition~\ref{prop:strFeller},
and we also have $R^U_\lambda \indicator_A=G^U_\lambda \indicator_A=0$
$M$-a.e.\ on $U$ by $\indicator_A=0$ $M$-a.e.
Since $M$ has full support, it follows that $R^U_\lambda \indicator_A$
is a continuous function on $U$ which is equal to $0$
on a dense subset of $U$ and hence it is identically zero on $U$,
proving the absolute continuity of $R^U_\lambda$.

\emph{Step 3:} Next we will show that for \emph{any} $x\in U$,
\begin{align} \label{eq:ident_laplace}
 \int_0^\infty e^{-\lambda t} \Bigl( \int_U p_t^U(x,y) f(y) \, M(dy) \Bigr) \, dt
	= \int_0^\infty e^{-\lambda t} E_x\bigl[f(\mathcal{B}^U_t)\bigr] \, dt,
\end{align}
for all $\lambda >0$ and all bounded Borel functions $f:U\to[0,\infty)$.
Recall that $P^U_t f(x)= E_x\bigl[f(\mathcal{B}^U_t)\bigr]$ denotes the
transition semigroup of $\mathcal{B}^U$. Then for any $\varepsilon >0$,
since $P^U_\varepsilon f=T^U_\varepsilon f$ $M$-a.e., by
the absolute continuity of $R^U_\lambda$ with respect to $M$ we have
\begin{align*}
 \int_{\varepsilon}^\infty e^{-\lambda t} P^U_t f(x) \, dt
	=e^{-\lambda \varepsilon} R^U_\lambda (P_\varepsilon^Uf)(x)
	&=e^{-\lambda \varepsilon} R^U_\lambda (T_\varepsilon^Uf)(x) \\
 =e^{-\lambda \varepsilon} R^U_\lambda \biggl( \sum_{n=1}^\infty e^{-\lambda_n \varepsilon} \langle \varphi_n, f\rangle \varphi_n \biggr)(x)
	&=\sum_{n=1}^\infty e^{-(\lambda+\lambda_n) \varepsilon} \frac{1}{\lambda+\lambda_n} \langle \varphi_n, f\rangle \varphi_n(x),
\end{align*} 
where we also used \eqref{eq:res_phi} and the uniform convergence of
the series in \eqref{eq:pU_eig}. Setting
$a^\varepsilon_n:=e^{-(\lambda+\lambda_n) \varepsilon} \frac{1}{\lambda+\lambda_n}
	=\int_\varepsilon^\infty e^{-(\lambda+\lambda_n)t} \, dt$
and applying dominated convergence again on the basis of the uniform convergence
of \eqref{eq:pU_eig} on $[\varepsilon,\infty)\times U\times U$, we further get
\begin{align*}
 \int_{\varepsilon}^\infty e^{-\lambda t} P^U_tf(x) \, dt
	=\sum_{n=1}^\infty a^\varepsilon_n \varphi_n(x) \langle \varphi_n, f\rangle
	&=\lim_{N\to \infty} \sum_{n=1}^N \int_\varepsilon^\infty e^{-(\lambda+\lambda_n)t} \varphi_n(x) \langle \varphi_n, f\rangle \, dt  \\
 =\int_\varepsilon^\infty \biggl( \sum_{n=1}^\infty e^{-\lambda_nt} \varphi_n(x) \langle \varphi_n, f\rangle \biggr) e^{-\lambda t}\, dt   
	&=\int_\varepsilon^\infty \Bigl( \int_U p^U_t(x,y) f(y) \, M(dy) \Bigr) e^{-\lambda t} \, dt,
\end{align*}
and we obtain \eqref{eq:ident_laplace} by using montone convergence to let $\varepsilon \downarrow 0$.

\emph{Step 4:} Finally, we now prove that
$P_x[\mathcal{B}^U_t \in dy]=p^U_t(x,y) \, M(dy)$ for all $(t,x)\in (0,\infty) \times U$.
Let $x\in U$. Applying to \eqref{eq:ident_laplace} the uniqueness of Laplace
transforms for positive measures on $[0,\infty)$
(see e.g.\ \cite[Section~XIII.1, Theorem~1a]{Fe71}),
we get for all bounded Borel functions $f:U\to[0,\infty)$,
\begin{align} \label{eq:ident_density}
 \int_U p^U_t(x,y) f(y) \, M(dy)=E_x\bigl[f(\mathcal{B}_t^U)\bigr]
	\qquad \text{for \emph{$dt$-a.e.}\ $t\in (0,\infty)$.}
\end{align}
If in addition $f$ is continuous, then we easily see from dominated convergence
using the continuity and boundedness of $p^U$ established in Step~1 that
\eqref{eq:ident_density} holds for all $t>0$.
Finally a monotone class argument gives the claim, proving i).

For ii), the claim is immediate from dominated convergence
in view of the continuity and boundedness of $p^U_t$ for each $t>0$
and the fact that $M(U)<\infty$.
Finally, iii) follows by \cite[Proposition~A.3 (2)]{Ka10}.
\end{proof}

\subsection{The heat kernel on unbounded open sets} \label{sec:Uunbounded}
The proof of Theorem~\ref{thm:dir_hk} for unbounded $U$ is based on the following
lemma, which essentially contains Theorem~\ref{thm:on-diag_glob} already.

\begin{lemma}\label{lem:off_diag_up}
For any $\beta> \alpha_1$ and any $R\geq 1$ there exist
$C_i=C_i(X,\gamma,R, \beta)>0$, $i=11,12$,
such that for any non-empty bounded open subset $U$ of $\bbR^2$,
\begin{align*}
 p^U_t(x,y)=p^U_t(y,x) \leq C_{11} t^{-1} \log( t^{-1})
	\exp\biggl(-C_{12} \Bigl(\frac{|x-y|^\beta \wedge 1} t \Bigr)^{\frac{1}{\beta -1}}\biggr) 
\end{align*}
for all $t\in (0, \tfrac 1 2]$, $x \in \bbR^2$ and $y \in B(R)$, where we extend
$p^U=p^U_t(x,y)$ to a function on $(0,\infty)\times \bbR^2 \times \bbR^2$
by setting $p^U_t(x,y):=0$ for $t>0$ and $(x,y)\in (U\times U)^c$.
\end{lemma}

\begin{proof}
Since for every $R\geq 1$ we have $p^{B(R)}_{t}\leq C_{10}t^{-1}\log(t^{-1})$
for any $t\in(0,\frac{1}{2}]$ for $C_{10}=C_{10}(X,\gamma,R)>0$ by
Proposition~\ref{prop:on_diag_up} and the continuity of $p^{B(R)}_{t}$,
given the exit time estimates in Proposition~\ref{prop:tailtau1},
the result follows from \cite[Theorem~1.1]{GK14}.
\end{proof}

\begin{remark} \label{rem:hk_bounded}
The constants appearing in the upper bound in Lemma~\ref{lem:off_diag_up}
do not depend on the set $U$. Therefore, for any $R\geq 1$ there exists
$C_{13}=C_{13}(X,\gamma, R)>0$, also not depending on $U$, such that
$p^U_{1/2}(x,y)\leq C_{13}$ for all $x\in \bbR^2$ and $y\in B(R)$.
In particular, by the semigroup property we have for all
$t\in (\frac 1 2, \infty)$ and such $x$ and $y$,
\begin{align*}
 p^U_t(x,y)=\int_{\bbR^2} p^U_{t- 1/2}(x,z) p^U_{1/2}(z,y) \, M(dz)
	\leq C_{13} \int_{\bbR^2} p^U_{t- 1/2}(x,z) \, M(dz) \leq C_{13}.
\end{align*}
\end{remark}

\begin{lemma} \label{lem:tail_tauUn}
For any increasing sequence $(U_n)_{n\geq 1}$ of open subsets of $\bbR^2$
satisfying $\bigcup_{n\geq 1} U_n = \bbR^2$,
\begin{align*}
 \lim_{n\to \infty} P_x[\tau_{U_n}<t]=0
\end{align*}
uniformly in $(t,x)$ over each compact subset of $[0,\infty)\times \bbR^2$.
\end{lemma}
\begin{proof}
It suffices to prove the uniform convergence in $(t,x)$ over $[0,T]\times B(R)$
for any $T,R \in (0,\infty)$. By monotonicity we may assume $t=T$.
Then for any $x\in B(R)$ and $n\geq 1$ with $B(2R) \subset U_n$, noting that
$\tau_{B(2R)}\leq \tau_{U_n}=\tau_{B(2R)}+\tau_{U_n}\circ\vartheta_{\tau_{B(2R)}}$,
by the strong Markov property \cite[Theorem~A.1.21]{CF12} of $\mathcal{B}$ we obtain
\begin{align}
 P_x[ \tau_{U_n}<T ]&= P_x\bigl[ \tau_{B(2R)}+\tau_{U_n}\circ\vartheta_{\tau_{B(2R)}} < T \bigr]
	\leq P_x \bigl[ \tau_{U_n} \circ \vartheta_{\tau_{B(2R)}} < T \bigr] \nonumber \\
 &= E_x\bigl[ P_{\mathcal{B}_{\tau_{B(2R)}}}[ \tau_{U_n}< T ] \bigr]
	= P_{\mu^x_{0,2R}} [ \tau_{U_n}< T ],
 \label{eq:tail_tauUn}
\end{align}
where $\mu^x_{0,2R}:=P_x[\mathcal{B}_{\tau_{B(2R)}} \in \cdot] = P_x[B_{T_{B(2R)}} \in \cdot]$
as in the proof of Proposition~\ref{prop:negmom_tau} above.
Setting $\mu_{0,2R}:=\mu^0_{0,2R}$ and arguing precisely as there,
from an explicit formula for the exit distribution of a Brownian motion
(see e.g.\ \cite[Exercise~4.2.24]{KS91})
we get $\mu^x_{0,2R}\leq c \mu_{0,2R}$
for some explicit $c>0$. Thus by \eqref{eq:tail_tauUn}, for any $n\geq 1$
with $B(2R) \subset U_n$ we obtain
$\sup_{x\in B(R)}P_x[ \tau_{U_n}<T ] \leq c P_{\mu_{0,2R}} [ \tau_{U_n}< T ]$,
which converges to $0$ as $n\to \infty$ by dominated convergence since the
trajectory of $\{\mathcal{B}_t\}_{t\in[0,T]}$ is bounded and hence contained
in $U_n$ for $n$ large enough, completing the proof.
\end{proof}

\begin{proof}[Proof of Theorem~\ref{thm:dir_hk} for unbounded $U$]
i) Let $R\geq 1$ and let $f: \bbR^2 \rightarrow [0,\infty)$ be bounded and
Borel measurable with $f|_{B(R)^c}=0$. Let $k,l\in\bbN$ satisfy $l>k\geq R+1$,
let $t>0$ and $x\in B(k)$. Noting that
$\tau_{B(k)}\leq\tau_{B(l)}=\tau_{B(k)}+\tau_{B(l)}\circ\vartheta_{\tau_{B(k)}}$
by $B(k) \subset B(l)$ and that
$E_x\bigl[\indicator_{\{\tau_{B(k)}=t\}}f(\mathcal{B}_{t})\bigr]=0$ by
$P_x[ \mathcal{B}_{\tau_{B(k)}}\in B(k)]=0$ and $f|_{B(k)^c}=0$, we see from
the strong Markov property \cite[Proposition~3.4]{GK14} of $\mathcal{B}$ that
\begin{align}\nonumber
 P^{B(l)}_t f(x)&=P^{B(k)}_tf(x)+ E_x\bigl[ \indicator_{\{\tau_{B(k)}< t < \tau_{B(l)}\}} f(\mathcal{B}_t) \bigr] \\
 &=P^{B(k)}_t f(x)+ E_x\bigl[ \indicator_{\{\tau_{B(k)}< t\}} P^{B(l)}_{t-\tau_{B(k)}}f(\mathcal{B}_{\tau_{B(k)}}) \bigr].
 \label{eq:hk_comp_SMarkov}
\end{align}
Recall that by Theorem~\ref{thm:dir_hk} i) for $U=B(l)$ proved in Subsection~\ref{sec:Ubounded},
\begin{align*}
 P^{B(l)}_{t-\tau_{B(k)}}f(\mathcal{B}_{\tau_{B(k)}})
	=\int_{B(R)} p^{B(l)}_{t-\tau_{B(k)}}(\mathcal{B}_{\tau_{B(k)}},y) f(y) \, M(dy).
\end{align*}
We have
\begin{align*}
\sup_{n\geq 1} \sup_{[\frac 1 2, \infty)\times \bbR^2 \times B(R)}
	p^{B(n)}_\cdot(\cdot,\cdot)\leq C_{13} < \infty
\end{align*}
by Remark~\ref{rem:hk_bounded} and
\begin{align*}
\sup_{n\geq 1} \sup_{(0,\frac  1 2)\times B(R+1)^c \times B(R)}
	p^{B(n)}_\cdot(\cdot,\cdot)\leq C_{14} < \infty
\end{align*}
for some $C_{14}=C_{14}(X,\gamma, R)>0$ by $\mathrm{dist}\bigl(B(R+1)^c, B(R)\bigr)\geq 1$
and Lemma~\ref{lem:off_diag_up}. We see therefore from \eqref{eq:hk_comp_SMarkov} that
\begin{align*}
 0\leq P^{B(l)}_t f(x)-P^{B(k)}_t f(x)
	\leq (C_{13}\vee C_{14}) P_x[\tau_{B(k)}< t] \int_{B(R)} f(y) \, M(dy),
\end{align*}
and since $f\geq 0$ with $f|_{B(R)^c}=0$ is arbitrary we obtain
\begin{align} \label{eq:hk_comp}
 0\leq p^{B(l)}_t(x,y)-p^{B(k)}_t(x,y)\leq (C_{13}\vee C_{14}) P_x[ \tau_{B(k)}<t ]
\end{align}
for all $t\in (0,\infty)$, $x\in B(k)$ and $y\in B(R)$ by virtue of the
continuity of $p^{B(l)}_{t}(x,\cdot)$ and $p^{B(k)}_t(x,\cdot)$ proved in
the last subsection. Thus it follows from \eqref{eq:hk_comp} and
Lemma~\ref{lem:tail_tauUn} that the limit
$p_t(x,y):=\lim_{n\to\infty}p^{B(n)}_t(x,y)\in[0,\infty)$ exists and is continuous
on $(0,\infty)\times \bbR^2 \times B(R)$. Since $R\geq 1$ is arbitrary and 
the relation $P_x[\mathcal{B}_t \in dy]=p_t(x,y) \, M(dy)$ can be obtained
from that for $\mathcal{B}^{B(n)}$ and $p^{B(n)}$ by monotone convergence,
statement i) follows  for the global heat kernel $p_t(x,y)$, i.e.\ for the case $U=\bbR^2$.
For a general unbounded open set $U\subset \bbR^2$, statement i) can be obtained
by similar arguments and the fact that for any $k,l\in \bbN$ with $k<l$,
\begin{align*} 
 0\leq p^{U\cap B(l)}_t(x,y)- p^{U\cap B(k)}_t(x,y)
	\leq p^{B(l)}_t(x,y)-p^{B(k)}_t(x,y), \qquad t>0,\,x,y\in \bbR^2.
\end{align*}
In order to see the latter inequality, notice that for
$(x,y)\in \bigl((U\cap B(k)) \times (U\cap B(k))\bigr)^c$ this inequality holds
trivially, and for $(x,y)\in (U\cap B(k)) \times (U\cap B(k))$,
\begin{align*}
 &p^{B(l)}_t(x,y)-p^{B(k)}_t(x,y)- p^{U\cap B(l)}_t(x,y)+ p^{U\cap B(k)}_t(x,y)\\
 &\mspace{200mu}=\lim_{r \downarrow 0}
	\frac{ P_x\bigl[ \mathcal{B}_t \in B(y,r), \, \tau_U \vee \tau_{B(k)} \leq t < \tau_{B(l)} \bigr]}{M\bigl((B(y,r)\bigr)}
	\geq 0
\end{align*}
by the continuity of the Dirichlet heat kernels on bounded open sets.

ii) Let $x\in U$ and $t,\varepsilon>0$. Since
\begin{align*}
 \int_{\bbR^2} p_t(x,y) \, M(dy)=P_x[\mathcal{B}_t\in \bbR^2]=1
\end{align*}
by $P_x[\lim_{s\to \infty} F_s=\infty]=1$, we can choose
$n\in \mathbb{N}$ such that $x\in B(n)$ and
\begin{align*}
 \int_{B(n)} p_t(x,y) \, M(dy) > 1-\varepsilon.
\end{align*}
Then by the continuity of $p_t$ there exists $r>0$ such that $B(x,r)\subset U$ and
\begin{align*}
 \int_{B(n)} p_t(z,y) \, M(dy) > 1-\varepsilon, \qquad \forall z\in B(x,r),
\end{align*}
and hence
\begin{align} \label{eq:strFeller_unbdd}
 \int_{U \setminus B(n)} p^U_t(z,y) \, M(dy) \leq \int_{B(n)^c} p_t(z,y) \, M(dy)< \varepsilon,
	\qquad \forall z\in B(x,r).
\end{align}
Now, for any bounded Borel function $f: U\rightarrow \bbR$ and $z\in B(x,r)$, writing
\begin{align*}
 P^U_t f(z)=\int_{U \setminus B(n)} p^U_t(z,y) f(y) \, M(dy)
	+ \int_{U \cap B(n)} p^U_t(z,y) f(y) \, M(dy)
\end{align*}
and applying \eqref{eq:strFeller_unbdd}, we obtain
\begin{align*}
 &\bigl| P^U_t f(x)- P^U_t f(z) \bigr| \\
 &\mspace{30mu}\leq   2 \|f\|_\infty \varepsilon + \biggl|\int_{U \cap B(n)} p^U_t(x,y) f(y) \, M(dy)
	- \int_{U \cap B(n)} p^U_t(z,y) f(y) \, M(dy) \biggr| \\
 &\mspace{30mu}\leq  (2 \|f\|_\infty+1) \varepsilon
\end{align*}
provided $|x-z|$ is sufficiently small, which proves the continuity of
$P^U_t f$ at $x$. In the last step we used the fact that, since $0\leq p^U_t \leq p_t$
on $B(x,r)\times (U\cap B(n))$ where $p_t$ is bounded and $p^U_t$ is continuous,
 the function $z\mapsto \int_{U \cap B(n)} p^U_t(z,y) f(y) \, M(dy)$
is continuous on $B(x,r)$ by dominated convergence.

iii) Since $U$ is connected, for any $x,y\in U$ there exists a connected bounded
open set $V\subset U$ with $x,y \in V$ and then by the corresponding result
for bounded open sets we have $p^U_t(x,y)\geq p^V_t(x,y)>0$ for any $t>0$.
\end{proof}

\begin{proof}[Proof of Theorem~\ref{thm:on-diag_glob}]
This is immediate from Lemma~\ref{lem:off_diag_up}
since, as shown in the above proof,
$p_t(x,y)=\lim_{n\to\infty}p^{B(n)}_t(x,y)$ for any $t>0$ and $x,y\in\bbR^2$.
\end{proof}

\section{On-diagonal lower bounds and spectral dimensions} \label{sec:lb}
In this section we prove the on-diagonal lower bound in Theorem~\ref{thm:on-diag_low}.
Indeed, we will show a more general result (Theorem~\ref{thm:on-diag_low_loc} below)
that also covers the Dirichlet Liouville heat kernels and thereby,
in combination with Theorem~\ref{thm:on-diag_glob},
enables us to identify the pointwise and global spectral dimensions as $2$.
Recall that we have fixed an environment $\om \in \Omega$
as declared at the beginning of Section~\ref{sec:cont_ub}.

\begin{theorem}\label{thm:on-diag_low_loc}
For $M$-a.e.\ $x\in \bbR^2$, for any $\eta>18$ and any open set $U\subset \bbR^2$
containing $x$ there exist $C_{15}=C_{15}(X,\gamma,|x|,\eta)>0$
and $t_0(x,U)=t_0(X,\gamma,\eta,x,U)\in (0,\frac{1}{2}]$ such that
\begin{align} \label{eq:hk_lb_loc}
 p_t^U(x,x)\geq C_{15} t^{-1} \bigl(\log(t^{-1})\bigr)^{-\eta},
	\qquad \forall t\in \bigl(0,t_0(x,U)\bigr].
\end{align}
\end{theorem}

In particular, Theorem~\ref{thm:on-diag_low_loc} immediately implies
Theorem~\ref{thm:on-diag_low} by choosing $U=\bbR^2$. Furthermore we can
deduce the following result on pointwise spectral dimension.

\begin{corro} \label{thm:spctr_dim_loc}
For $M$-a.e.\ $x\in\bbR^2$, for any open set $U\subset \bbR^2$ containing $x$,
\begin{align} \label{eq:spr_dim}
 \lim_{t\downarrow 0} \frac{2 \log p^U_t(x,x)}{-\log t}=2.
\end{align}
\end{corro}

\begin{proof}
This is immediate from the lower bound in Theorem~\ref{thm:on-diag_low_loc}
and the on-diagonal part of the upper bound in Theorem~\ref{thm:on-diag_glob}
together with $p^U_t(x,x)\leq p_t(x,x)$.
\end{proof}

The proof of Theorem~\ref{thm:on-diag_low_loc} is given in
Subsection~\ref{sec:proof_on-diag_low}, and then the application to the
identification of the global spectral dimension is presented in
Subsection~\ref{sec:spctr_dim}.

\subsection{Proof of Theorem~\ref{thm:on-diag_low_loc}} \label{sec:proof_on-diag_low}

In order to show Theorem~\ref{thm:on-diag_low_loc} we need further moment
and tail estimates on the exit times from balls. First, we recall the
representation of the expected exit time in terms of the Green kernel.

\begin{lemma} \label{lem:mom_tau_Green}
For any non-empty open set $U\subset \bbR^2$ and any $x\in U$,
\begin{align*}
 E_x[\tau_U]=\int_U g_U(x,y) \, M(dy).
\end{align*}
\end{lemma}
\begin{proof}
This follows immediately from Proposition~\ref{prop:corrF_M}.
\end{proof}

\begin{lemma} \label{lem:mom_tau_up}
For any  $R\geq 1$ 
there exist $c_{1}=c_{1}(\gamma)>0$ and $C_{16}=C_{16}(X,\gamma, R)>0$ such that
\begin{align*}
 E_x[\tau_{B(x,r)}] \leq M\bigl(B(x,r)\bigr)\bigl(C_{16}+c_{1}\log(r^{-1})\bigr),
	\qquad \forall x\in B(R), \, r\in (0,1].
\end{align*}
\end{lemma}
\begin{proof}
Since $g_{B(x,r)}\leq g_{B(R+2)}$ by $B(x,r)\subset B(R+2)$,
we see from Lemma~\ref{lem:mom_tau_Green}, \eqref{eq:killed_green}
and $B(x,r)\subset B(R+1)$ that
\begin{align*}
 E_x[\tau_{B(x,r)}] =\int_{B(x,r)} g_{B(x,r)}(x,y) \, M(dy)
	&\leq \int_{B(x,r)} g_{B(R+2)}(x,y) \, M(dy)\\
 &\leq \int_{B(x,r)} \Bigl( \frac 1 \pi \log \frac 1 {|x-y|} + c \Bigr)\,M(dy)
\end{align*}
with $c=c(R)>0$.
Setting $D_n(x):=B(x,2^{1-n}r) \setminus B(x,2^{-n}r)$ for $n\geq 1$ and
noting that $M(\{x\})=0$ by Lemma~\ref{lem:vol_dec}, we further obtain
\begin{align} \label{eq:mom_tau_up}
 E_x[\tau_{B(x,r)}] \leq c M\bigl(B(x,r)\bigr)
	+\frac{1}{\pi}\sum_{n=1}^\infty \bigl(n+\log(r^{-1})\bigr) M\bigl(D_n(x)\bigr).
\end{align}
On the other hand, Lemma~\ref{lem:vol_dec} implies that for $\varepsilon:=\alpha_{2}/2$,
\begin{align*}
 \bigl(n+\log(r^{-1})\bigr) \frac{M\bigl(B(x,2^{1-n}r)\bigr)}{M\bigl(B(x,r)\bigr)}
	\leq C n 2^{-n(\alpha_2-\varepsilon)} r^{\alpha_2-\alpha_1-2\varepsilon}
	\leq C 2^{-n\alpha_{2}/4}
\end{align*}
provided $n\geq c \log(r^{-1})$ with $c=c(\gamma)>0$, which together with
\eqref{eq:mom_tau_up} yields
\begin{align*}
 E_x[\tau_{B(x,r)}]&\leq cM\bigl(B(x,r)\bigr) +C \sum\nolimits_{n\geq c \log(r^{-1})} 2^{-n\alpha_{2}/4} M\bigl(B(x,r)\bigr) \\
 &\mspace{160mu}+ cM\Bigl(\bigcup\nolimits_{1\leq n < c \log(r^{-1})}D_n(x)\Bigr) \log(r^{-1})\\
 &\leq M\bigl(B(x,r)\bigr)\bigl(C+c\log(r^{-1})\bigr),
\end{align*}
completing the proof.
\end{proof}

\begin{lemma} \label{lem:mom_tau_low}
There exists a constant $c_2>0$ such that 
\begin{align*}
 E_x[\tau_{B(x,r)}] \geq c_2 M\bigl(B(x,r/2)\bigr),
	\qquad \forall x\in \bbR^2, \, r>0.
\end{align*}
\end{lemma}
\begin{proof}
By using Lemma~\ref{lem:mom_tau_Green} and the translation and scale
invariance of the Green kernel
\begin{align*}
 g_{B(x,r)}(x,y)=g_{B(0,r)}(0,y-x)=g_{B(0,\lambda r)}\bigl(0,\lambda (y-x)\bigr)
\end{align*}
for $x,y\in \bbR^2$, $r>0$ and $\lambda>0$
(see e.g.\ \cite[Example~1.5.1]{FOT11}), we obtain
\begin{align*}
 E_x[\tau_{B(x,r)}] &=\int_{B(x,r)} g_{B(x,r)}(x,y) \, M(dy) \\
	&\geq\int_{B(x,r/2)} g_{B(0,1)}\bigl(0,r^{-1}(y-x)\bigr) \, M(dy) 
 \geq c_2 M\bigl(B(x,r/2)\bigr)
\end{align*}
with $c_2:=\inf_{y\in B(0,1/2)}g_{B(0,1)}(0,y)>0$, which is the claim.
\end{proof}

\begin{prop} \label{prop:tailtau2}
For any $R\geq 1$ 
there exists  $C_{17}=C_{17}(X, \gamma, R)>0$ such that 
\begin{align*}
 P_x[\tau_{B(x,r)}\leq t]
	\leq 1-C_{17} \frac{M\bigl(B(x,r/2)\bigr)}{M\bigl(B(x,3r)\bigr)\log(r^{-1})}
\end{align*}
for all $x\in B(R)$, $ r\in (0,\tfrac 1 2]$ and $0<t\leq \tfrac 1 2 E_x[\tau_{B(x,r)}]$.
\end{prop}
\begin{proof}
For any $t>0$, by the obvious relation
$\tau_{B(x,r)}\leq t+\indicator_{\{\tau_{B(x,r)}>t\}}(\tau_{B(x,r)}-t)
	=t+\indicator_{\{\tau_{B(x,r)}>t\}}(\tau_{B(x,r)}\circ\vartheta_{t})$
and the Markov property \cite[Theorem~A.1.21]{CF12} of $\mathcal{B}$,
\begin{align*}
 E_x[\tau_{B(x,r)}] &\leq t +E_x \bigl[ \indicator_{\{\tau_{B(x,r)}>t\}}(\tau_{B(x,r)}\circ\vartheta_{t}) \bigr]
	=t +E_x \bigl[ \indicator_{\{\tau_{B(x,r)}>t\}} E_{\mathcal{B}_t}[\tau_{B(x,r)}] \bigr] \\
 &\leq  t + P_x[\tau_{B(x,r)} >t] \sup_{y\in B(x,r)} E_y[\tau_{B(x,r)}],
\end{align*}
which implies that for $0<t\leq \tfrac 1 2 E_x[\tau_{B(x,r)}]$,
\begin{align} \label{eq:tailtau2}
 P_x[\tau_{B(x,r)} \leq t]\leq 1+ \frac{t-E_x[\tau_{B(x,r)}] }{\sup_{y\in B(x,r)} E_y[\tau_{B(x,r)}]}
	\leq 1- \frac{\tfrac 1 2 E_x[\tau_{B(x,r)}] }{ \sup_{y\in B(x,r)} E_y[\tau_{B(x,r)}]}.
\end{align}
Then since $B(x,r) \subset B(y,2r) \subset B(x,3r)$ and hence
$\tau_{B(x,r)}\leq\tau_{B(y,2r)}$ for any $y\in B(x,r)$,
from Lemma~\ref{lem:mom_tau_up} we obtain
\begin{align*}
 \sup_{y\in B(x,r)} E_y[\tau_{B(x,r)}] \leq \sup_{y\in B(x,r)} E_y[\tau_{B(y,2r)}]
	&\leq C \sup_{y\in B(x,r)} M\bigl(B(y,2r)\bigr)  \log(r^{-1}) \\
 &\leq  C M\bigl(B(x,3r)\bigr) \log(r^{-1}),
\end{align*}
and the claim follows by applying this estimate and
Lemma~\ref{lem:mom_tau_low} to \eqref{eq:tailtau2}.
\end{proof}

We are now in the position to show an on-diagonal lower bound on
the Dirichlet Liouville heat kernels.
\begin{prop} \label{prop:lowbound_dirHK}
For any $R\geq 1$ 
there exists $C_{18}=C_{18}(X, \gamma, R)>0$ such that 
\begin{align*}
 p^{B(x,r)}_t(x,x)
	\geq \frac{C_{18}}{M\bigl(B(x,r)\bigr)}
		\biggl( \frac{M\bigl(B(x,r/2)\bigr)}{M\bigl(B(x,3r)\bigr) \log(r^{-1})} \biggr)^2
\end{align*}
for all $x \in B(R)$, $r\in (0,\frac 1 2]$ and  $0 <t\leq c_2 M\bigl(B(x,r/2)\bigr)$
(with $c_2$ as in Lemma~\ref{lem:mom_tau_low}).
\end{prop}
\begin{proof}
Let $0<t\leq\frac{1}{2} c_2 M\bigl(B(x,r/2)\bigr)$. Since
$\frac{1}{2}E_x[\tau_{B(x,r)}] \geq \frac{1}{2}c_2 M\bigl(B(x,r/2)\bigr)\geq t$
by Lemma~\ref{lem:mom_tau_low}, we see from Proposition~\ref{prop:tailtau2},
the Cauchy-Schwarz inequality and the symmetry and semigroup property
of the Dirichlet heat kernel $p^{B(x,r)}$ that
\begin{align*}
 &\biggl(C_{17} \frac{M\bigl(B(x,r/2)\bigr)}{M\bigl(B(x,3r)\bigr)\log(r^{-1})}\biggr)^{2}
	\leq P_x[\tau_{B(x,r)}>t]^2 \\
 &\mspace{60mu} = P_x[\mathcal{B}_t\in B(x,r), \,\tau_{B(x,r)}>t]^2
	= \Bigl( \int_{B(x,r)} p^{B(x,r)}_t (x,y) \, M(dy) \Bigr)^2 \\
 &\mspace{60mu} \leq M\bigl(B(x,r)\bigr) \int_{B(x,r)} \bigl(p^{B(x,r)}_{t}(x,y)\bigr)^2 \, M(dy)
	= M\bigl(B(x,r)\bigr) p_{2t}^{B(x,r)}(x,x),
\end{align*}
which gives the result.
\end{proof}

\begin{corro} \label{cor:spctr_dim}
Let $c_3>0$, $x\in \bbR^2$, $\eta>18$ and set $\kappa:=\frac{1}{8}(\eta-2)$.
If $r_0\in(0,\frac{1}{2}]$ and
\begin{align} \label{cond:vol_growth}
 M\bigl(B(x,2r)\bigr) \leq c_3 \bigl(\log(r^{-1})\bigr)^{\kappa} M\bigl(B(x,r)\bigr),
	\qquad \forall r \in (0, r_0],
\end{align}
then for any open set $U\subset \bbR^2$ containing $x$ there exist
$C_{15}=C_{15}(X,\gamma,|x|,\eta,c_3)>0$ and
$t_0(x,U)=t_0(X,\gamma,x,U,r_0)\in (0,\frac{1}{2}]$
such that \eqref{eq:hk_lb_loc} holds.
\end{corro}
\begin{proof}
Let $U$ be an open subset of $\bbR^2$ with $x\in U$ and let
$r_1=r_1(x,U,r_0)\in(0,r_0/2]$ be such that $B(x,r_1)\subset U$. Also,
noting that $\lim_{r\downarrow 0}M\bigl(B(x,r)\bigr)=M(\{x\})=0$ by Lemma~\ref{lem:vol_dec},
for $0<t\leq c_2 M\bigl(B(x,r_{1}/2)\bigr)$ let $n=n(t)\geq 1$ be such that
\begin{align*}
c_2 M\bigl(B(x,2^{-n-1}r_1)\bigr)< t\leq c_2 M\bigl(B(x,2^{-n}r_1)\bigr)
\end{align*}
and set $r=r(t):=2^{1-n}r_1$. Then by Proposition~\ref{prop:lowbound_dirHK},
\begin{align} \label{eq:est_t_p}
 t p^U_t(x,x)&\geq t p^{B(x,r_1)}_t(x,x)\geq t p^{B(x,r)}_t(x,x) \nonumber \\
 &\geq  C \frac{M\bigl(B(x,r/4)\bigr)}{M\bigl(B(x,r)\bigr)} \biggl( \frac{M\bigl(B(x,r/2)\bigr)}{M\bigl(B(x,3r)\bigr)} \biggr)^2
	\bigl(\log(r^{-1})\bigr)^{-2}.
\end{align}
On the other hand, we see from \eqref{cond:vol_growth} that
\begin{align*}
 \frac{M\bigl(B(x,r/4)\bigr)}{M\bigl(B(x,r)\bigr)}
	= \frac{M\bigl(B(x,r/4)\bigr)}{M\bigl(B(x,r/2)\bigr)} \frac{M\bigl(B(x,r/2)\bigr)}{M\bigl(B(x,r)\bigr)}
	\geq c \bigl(\log(r^{-1})\bigr)^{-2\kappa}
\end{align*}
and
\begin{align*}
 \frac{M\bigl(B(x,r/2)\bigr)}{M\bigl(B(x,3r)\bigr)}
	\geq \frac{M\bigl(B(x,r/2)\bigr)}{M\bigl(B(x,4r)\bigr)}
	\geq c \bigl(\log(r^{-1})\bigr)^{-3\kappa}
\end{align*}
with $c=c(c_3,\eta)>0$. Now \eqref{eq:hk_lb_loc} follows by combining these
esimates with \eqref{eq:est_t_p} and noting that
$c\log(t^{-1})\leq\log(r^{-1})=\log\bigl(r(t)^{-1}\bigr)\leq c'\log(t^{-1})$
with $c=c(\gamma)>0$ and $c'=c'(\gamma)>0$ provided $t\leq t_{0}'$ for some
$t_{0}'=t_{0}'(X,\gamma,|x|)\in(0,\frac{1}{2}]$ by Lemma~\ref{lem:vol_dec}.
\end{proof}

Now Theorem~\ref{thm:on-diag_low_loc} follows by Lemma~\ref{lem:vol_dec},
Corollary~\ref{cor:spctr_dim} and the following result.

\begin{prop} \label{prop:gen_volasymp}
Let $\mu$ be a Borel measure on $\bbR^2$ satisfying $\mu\bigl(B(x,r)\bigr) \in (0,\infty)$
for all $x \in \bbR^2$ and $r>0$. Then for $\mu$-a.e.\ $x\in \bbR^2$, for any
$\kappa>2$ there exists $r_0(x)=r_0(\mu,\kappa,x)\in(0,\frac 1 2]$ such that
\begin{align} \label{eq:gen_volasymp}
 \mu\bigl(B(x,2r)\bigr) \leq 8 \bigl(\log(r^{-1})\bigr)^\kappa \mu\bigl(B(x,r)\bigr),
	\qquad \forall r \in (0, r_0(x)].
 \end{align}
\end{prop}
\begin{proof}
Since \eqref{eq:gen_volasymp} is weaker for larger $\kappa$, it suffices to show
\eqref{eq:gen_volasymp} for $\mu$-a.e.\ $x\in \bbR^2$ for each $\kappa\in(2,\frac{5}{2}]$.
Fix an arbitrary $x_0 \in \bbR^2$. Set $r_k:=2^{-k}$ for $k\in \bbZ$,
$\mu_{x_0}:=\mu\bigl(\cdot \cap B(x_0,1)\bigr)$ and, for $n\in \bbN$,
\begin{align*}
 A_n&:=\bigl\{ x\in B(x_0,1): \, \mu\bigl(B(x,r_{n-1})\bigr) \geq n^{\kappa/2} \mu\bigl(B(x,r_{n})\bigr) \bigr\}, \\
 \Xi_n&:= \bigl\{x_0+\bigl(\tfrac k {2^n}, \tfrac l {2^n}\bigr) :\, k,l\in \bbZ,\,|k|,|l|\leq 2^n \bigr\}.
\end{align*}
Then since $B(x_0,1)\subset \bigcup_{x\in \Xi_{n+1}} B(x,r_{n+1})$ and furthermore
$B(x,r_{n+1})\subset B(y,r_{n})\subset B(y,r_{n-1})\subset B(x,r_{n-2})$
for $x\in \Xi_{n+1}$ and $y\in B(x,r_{n+1})$, we have
\begin{align*}
 \int_{B(x_0,1)} \frac{\mu\bigl(B(y,r_{n-1})\bigr)}{\mu\bigl(B(y,r_{n})\bigr)} \, \mu_{x_0}(dy)
	&\leq \sum_{x\in \Xi_{n+1}} \int_{B(x,r_{n+1})} \frac{\mu\bigl(B(y,r_{n-1})\bigr)}{\mu\bigl(B(y,r_{n})\bigr)} \, \mu(dy) \\
 &\leq \sum_{x\in \Xi_{n+1}} \int_{B(x,r_{n+1})} \frac{\mu\bigl(B(x, r_{n-2})\bigr)}{\mu\bigl(B(x,r_{n+1})\bigr)} \, \mu(dy) \\
 &= \int_{\bbR^2} \sum_{x\in \Xi_{n+1}} \indicator_{B(x,r_{n-2})}(y) \, \mu(dy)
	\leq c \mu\bigl(B(x_0,4)\bigr)
\end{align*}
for some $c>0$. By \v{C}eby\v{s}ev's inequality this implies
$\mu_{x_0}(A_n)\leq c \mu\bigl(B(x_0,4)\bigr) n^{-\kappa/2}$,
hence $\sum_{n=1}^{\infty} \mu_{x_0}(A_n)<\infty$, and therefore
by the Borel-Cantelli lemma for $\mu$-a.e.\ $x\in B(x_0,1)$ there exists
$n_0(x)=n_0(\mu,\kappa,x)\in \bbN$ such that
\begin{align} \label{eq:muvol}
 \mu\bigl(B(x,r_{n-1})\bigr)\leq n^{\kappa/2} \mu\bigl(B(x,r_{n})\bigr),
	\qquad \forall n\geq n_0(x).
\end{align}
Now let $x\in B(x_0,1)$ satisfy \eqref{eq:muvol}, let $r\in (0,r_{n_0(x)}]$
and let $n\geq n_0(x)$ be such that $r_{n+1}<r\leq r_{n}$.
Then by applying \eqref{eq:muvol} twice,
\begin{align*}
 \mu\bigl(B(x,2r)\bigr)\leq \mu\bigl(B(x,r_{n-1})\bigr)
	\leq n^{\kappa/2} (n+1)^{\kappa/2} \mu\bigl(B(x,r_{n+1})\bigr)
	\leq 2^{3/2} n^{\kappa} \mu\bigl(B(x,r)\bigr)
\end{align*}
with $ n\leq\frac{1}{\log 2}\log(r^{-1})$.
Finally, since $x_0$ is arbitrary, the claim follows.
\end{proof}

\subsection{Global spectral dimension} \label{sec:spctr_dim}
Let $U \subset \bbR^2$ be non-empty, open and bounded.
As in Section~\ref{sec:Ubounded} above, let $\bigl(\lambda_n(U)\bigr)_{n\geq 1}$
be the eigenvalues of $-\mathcal{L}_U$ written in increasing order and
repeated according to multiplicity, and define
\begin{align*}
 Z_U(t):=\int_U p^U_t(x,x) \, M(dx)=\sum_{n=1}^\infty e^{-\lambda_n(U) t},
	\qquad t>0.
\end{align*}
Then we obtain the following estimates of $Z_{U}(t)$ from
Theorems~\ref{thm:on-diag_glob} and \ref{thm:on-diag_low_loc} and
conclude in particular that the global spectral dimension is $2$.

\begin{corro} \label{cor:trace}
Let $R\geq 1$ and let $U\subset B(R)$ be a non-empty open subset of $\mathbb{R}^{2}$.
Then for any $\eta>18$ there exist $C_{19}=C_{19}(X,\gamma,R)>0$,
$C_{20}=C_{20}(X,\gamma,R,\eta)>0$ and
$t_{1}(U)=t_{1}(X,\gamma,\eta,U)\in(0,\frac{1}{2}]$ such that
\begin{align} \label{eq:trace_upper}
 Z_{U}(t)&\leq C_{19}M(U)t^{-1}\log(t^{-1}), \qquad\qquad \forall t\in (0,\tfrac{1}{2}],\\
 Z_{U}(t)&\geq C_{20}M(U)t^{-1}\bigl(\log(t^{-1})\bigr)^{-\eta}, \qquad \forall t\in (0,t_{1}(U)].
\label{eq:trace_lower}
\end{align}
In particular,
\begin{align} \label{eq:glob_spctr_dim}
 \lim_{t \downarrow 0} \frac{2\log Z_U(t)}{-\log t}=2.
\end{align}
\end{corro}

\begin{proof}
\eqref{eq:glob_spctr_dim} is a direct consequence of \eqref{eq:trace_upper} and
\eqref{eq:trace_lower}, and \eqref{eq:trace_upper} is immediate from the inequality
$p^{U}_{t}(x,x)\leq p_{t}(x,x)$ and the on-diagonal part of the upper bound in
Theorem~\ref{thm:on-diag_glob}. Thus it remains to verify \eqref{eq:trace_lower}.
We may assume that $R=R(U):=\sup_{x\in U}|x|$. Let $\eta>18$, let
$C_{15}=C_{15}(X,\gamma,R,\eta)>0$ be as in Theorem~\ref{thm:on-diag_low_loc}
and define an upper semi-continuous function $t_{U}:U\to[0,\frac{1}{2}]$ by
\begin{align}\label{eq:tU-dfn}
 t_{U}(x):=\inf\bigl\{t\in(0,\tfrac{1}{2}]:\,
	t\bigl(\log(t^{-1})\bigr)^{\eta}p^{U}_{t}(x,x)<C_{15}\bigr\}
	\qquad(\inf\emptyset:=\tfrac{1}{2}).
\end{align}
Then $t_{U}(x)>0$ for $M$-a.e.\ $x\in U$ by Theorem~\ref{thm:on-diag_low_loc}
and therefore there exists $t_{1}=t_{1}(X,\gamma,\eta,U)\in(0,\frac{1}{2}]$
such that $M\bigl(t_{U}^{-1}([t_{1},\frac{1}{2}])\bigr)\geq\frac{1}{2}M(U)$.
Now for each $t\in(0,t_{1}]$,
$p^{U}_{t}(x,x)\geq C_{15}t^{-1}\bigl(\log(t^{-1})\bigr)^{-\eta}$
for any $x\in t_{U}^{-1}([t_{1},\frac{1}{2}])$ by $t\leq t_{1}\leq t_{U}(x)$
and \eqref{eq:tU-dfn}, and hence
\begin{align*}
 Z_{U}(t)\geq\int_{t_{U}^{-1}([t_{1},\frac{1}{2}])}p^{U}_{t}(x,x)\,M(dx)
	&\geq C_{15}t^{-1}\bigl(\log(t^{-1})\bigr)^{-\eta}M\bigl(t_{U}^{-1}([t_{1},\tfrac{1}{2}])\bigr)\\
 &\geq\tfrac{1}{2}C_{15}M(U)t^{-1}\bigl(\log(t^{-1})\bigr)^{-\eta},
\end{align*}
proving \eqref{eq:trace_lower}.
\end{proof}

\begin{remark}
It is unknown to the authors whether the eigenvalue counting function
$N_U(\lambda):=\# \{n\in \bbN:\, \lambda_n(U) \leq \lambda \}$ satisfies the
counterparts of \eqref{eq:trace_upper}, \eqref{eq:trace_lower} and \eqref{eq:glob_spctr_dim}.
\end{remark}

\appendix

\section{Proof of Proposition~\ref{prop:pcaf}} \label{app:pcaf}

The proof will be based on Lemma~\ref{lem:Fn_Tsmall} and
the following result proved in \cite{GRV13}.

\begin{theorem} \label{thm:constrF}
For each $x\in \bbR^2$,  $\bbP\times P_x$-a.s.\ the following hold:
\begin{enumerate}
 \item[i)] For all $t \geq 0$, $F_t:=\lim_{n\to \infty} F^n_t$ exists in $\bbR$.
 \item[ii)] The mapping $[0,\infty)\ni t\mapsto F_t\in [0,\infty)$ is continuous,
	strictly increasing and satisfies $F_0=0$ and $\lim_{t\to \infty} F_t=\infty$.
\end{enumerate}
\end{theorem}
\begin{proof}
See \cite[Lemma~2.8 and Proof of Theorem~2.7]{GRV13}.
\end{proof}

We start with a preparatory lemma.

\begin{lemma} \label{lem:liminfF}
$\bbP$-a.s., for all $x\in \bbR^2$,
\begin{align*}
 \lim_{t\downarrow 0}\liminf_{n\to\infty} F^n_t=0 \qquad \text{$P_x$-a.s.}
\end{align*}
\end{lemma}
\begin{proof}
Fix any environment $\om \in \Omega$ such that the conclusion of
Lemma~\ref{lem:Fn_Tsmall} holds, and let $x\in\bbR^2$. Then since $F^n_t$ for
$n\in\bbN$ are non-decreasing in $t$ and hence so is $\liminf_{n\to\infty} F^n_t$,
we see from Fatou's lemma and Lemma~\ref{lem:Fn_Tsmall} that
\begin{align*}
 0\leq E_x\Bigl[\lim_{t\downarrow 0}\liminf_{n\to\infty} F^n_t\Bigr]
	\leq \lim_{t\downarrow 0} E_x\Bigl[\liminf_{n\to\infty} F^n_t\Bigr]
	\leq \lim_{t\downarrow 0}\liminf_{n\to \infty} E_x[F^n_t]
	=0,
\end{align*}
which implies that $\lim_{t\downarrow 0}\liminf_{n\to\infty} F^n_t=0$ $P_x$-a.s.
\end{proof}

For each $t\geq 0$ we denote by $\Lambda_t$ the set of all
$(\om,\om')\in \Omega \times \Omega'$ such that:
\begin{enumerate}
 \item [i)] For all $u\in[t,\infty)$,
	$F_{t,u}(\om,\om'):=\lim_{n\to \infty}\bigl(F^n_{u}(\om,\om')-F^n_t(\om,\om')\bigr)$
	exists in $\bbR$.
 \item[ii)] The mapping $[t,\infty)\ni u\mapsto F_{t,u}(\om,\om')\in [0,\infty)$
	is continuous, strictly increasing and satisfies $F_{t,t}(\om,\om')=0$
	and $\lim_{u\to \infty} F_{t,u}(\om,\om')=\infty$.
\end{enumerate}
We also set $\Lambda_{t}^{\om}:=\{\om' \in \Omega':\, (\om,\om')\in \Lambda_t\}$
for $\om \in \Omega$. Note that $\Lambda_{t}^{\om}=\theta_t^{-1}(\Lambda_{0}^{\om})$
thanks to the fact that for all $n\in\bbN$ and $\om'\in\Omega'$,
\begin{align} \label{eq:Fn_additive}
 F^n_{s+t}(\om,\om')=F^n_t(\om,\om')+F^n_s\bigl(\om,\theta_t(\om')\bigr),
	\qquad \forall s,t\geq 0.
\end{align}
Furthermore we have $\Lambda_t \in \mathcal{A} \otimes \mathcal{G}^0_\infty$, since
$F^n_s$ is $\mathcal{A} \otimes \mathcal{G}^0_s$-measurable for any $n\in\bbN$
and $s\geq 0$ and $\Lambda_t$ is easily seen to be equal to
\begin{align*}
 \left\{ (\om,\om') \in \Omega\times\Omega':\,
	\begin{minipage}{290pt}
	$F_{t,s+t}(\om,\om'):=\lim_{n\to \infty}\bigl(F^n_{s+t}(\om,\om')-F^n_t(\om,\om')\bigr)$
	exists in $\bbR$ for all $s\in\mathbb{Q}\cap [0,\infty)$,
	$\mathbb{Q}\cap[0,N]\ni s\mapsto F_{t,s+t}(\om,\om')\in[0,\infty)$
	is uniformly continuous and strictly increasing for any $N\in\bbN$,
	$\lim_{\mathbb{Q}\ni s\to\infty}F_{t,s+t}(\om,\om')=\infty$
	\end{minipage}
	\right\}
\end{align*}
by virtue of the monotonicity of $F^n_s$ in $s$. Finally, recall that
$\bbP\times P_x[\Lambda_0]=1$ for all $x\in \bbR^2$ by Theorem~\ref{thm:constrF}.

\begin{lemma} \label{lem:lambda_t}
For $\bbP$-a.e.\ $\om\in\Omega$, $ P_x[\Lambda^\om_t]=1$
for all $t>0$ and $x\in \bbR^2$.
\end{lemma}

\begin{proof}
Let $\mu(dy):=(2\pi)^{-1}e^{-|y|^{2}/2}\,dy$. By Fubini's theorem, we have
$\bbE P_x[\Lambda_0^\om]=\bbP\times P_x[\Lambda_0]=1$ for all $x\in \bbR^2$
and then its $\mu(dx)$-integral results in $\bbE P_\mu[\Lambda_0^\om]=1$ with
$P_\mu[\cdot]:=\int_{\bbR^2} P_x[\cdot] \, \mu(dx)$. Thus
for $\bbP$-a.e.\ $\om\in\Omega$, $P_\mu[(\Lambda_0^\om)^c]=0$, namely
\begin{align} \label{eq:Py_zero}
 P_y[(\Lambda_0^\om)^c]=0 \qquad \text{for $dy$-a.e.\ $y\in \bbR^2$}.
\end{align}
Now for any such $\omega\in\Omega$ and for all $t>0$ and $x\in \bbR^2$, we have
\begin{align*}
 P_x[\Lambda_t^\om]=P_x[\theta_t^{-1}(\Lambda_0^\om)]
	=E_x[\indicator_{\Lambda_0^\om}\circ \theta_t]
	=E_x\bigl[ P_{B_t}[\Lambda_0^\om] \bigr]
	=\int_{\bbR^2} P_y[\Lambda_0^\om] q_t(x,y) \, dy=1
\end{align*}
by the Markov property of $B$ and \eqref{eq:Py_zero}, completing the proof.
\end{proof}

\begin{proof}[Proof of Proposition~\ref{prop:pcaf}]
Set $\mathbb{Q}_+ :=\mathbb{Q}\cap(0,\infty)$ and
\begin{align*}
 \Lambda:=\Bigl\{ (\om,\om')\in \Omega\times \Omega':\,
	\lim_{t\downarrow 0}\liminf_{n\to \infty} F^n_t(\om,\om')=0 \Bigr\}
	\cap \bigcap_{q\in \mathbb{Q}_+} \Lambda_{q}.
\end{align*}
Then clearly $\Lambda\in\mathcal{A} \otimes \mathcal{G}^0_\infty$, and
i) follows immediately from Lemmas~\ref{lem:liminfF} and \ref{lem:lambda_t}.

Let $(\om,\om')\in\Lambda$. Then for each $q\in \mathbb{Q}_+$,
$(\om,\om')\in \Lambda_q$, so that for all $t\in[q,\infty)$ the limit $F_{q,t}(\om,\om')$
exists in $\bbR$, $[q,\infty)\ni t\mapsto F_{q,t}(\om,\om')\in[0,\infty)$ is continuous
and strictly increasing and $\lim_{t\to \infty} F_{q,t}(\om,\om')=\infty$.
Thus for all $0<s\leq t$ the limit
\begin{align} \label{eq:FstFqtFqs}
 F_{s,t}(\om,\om'):=\lim_{n\to \infty}\bigl(F^n_{t}(\om,\om')-F^n_{s}(\om,\om')\bigr)
	=F_{q,t}(\om,\om')-F_{q,s}(\om,\om')
\end{align}
exists in $\bbR$, where $q\in\mathbb{Q}\cap(0,s]$, and
$[s,\infty)\ni t\mapsto F_{s,t}(\om,\om')\in[0,\infty)$ is a strictly increasing
continuous function satisfying $\lim_{t\to\infty}F_{s,t}(\om,\om')=\infty$.
Moreover, for any $t>0$ and $0 < u\leq s\leq t$,
\begin{align*}
 0\leq F_{u,t}(\om,\om')-F_{s,t}(\om,\om')
	= \lim_{n\to \infty} \bigl( F^n_{s}(\om,\om') - F^n_{u}(\om,\om') \bigr)
 \leq \liminf_{n\to \infty}  F^n_{s}(\om,\om'), 
\end{align*}
which tends to $0$ as $s\downarrow 0$ and thereby verifies Cauchy's convergence
criterion for $\bigl(F_{s,t}(\om,\om')\bigr)_{s\in(0,t]}$ as $s\downarrow 0$.
Hence the finite limit $F_t(\om,\om') :=\lim_{s\downarrow 0} F_{s,t}(\om,\om')$
exists, and then recalling \eqref{eq:FstFqtFqs}, we easily obtain
\begin{align} \label{eq:Ft_at_0}
 0\leq F_t(\om,\om') 
	=\lim_{s\downarrow 0}\lim_{n\to\infty} \bigl( F^n_{t}(\om,\om') - F^n_{s}(\om,\om') \bigr)
	\leq \liminf_{n\to\infty} F^n_t(\om,\om')
	\xrightarrow{t\downarrow 0}0
\end{align}
and, for all $0<s\leq t$,
\begin{align} \label{eq:FtFsFst}
 F_t(\om,\om')-F_s(\om,\om')
	=\lim_{u\downarrow 0}\bigl(F_{u,t}(\om,\om')-F_{u,s}(\om,\om')\bigr)
	=F_{s,t}(\om,\om').
\end{align}
Now by \eqref{eq:Ft_at_0}, \eqref{eq:FtFsFst} and the properties of the
function $t\mapsto F_{s,t}(\om,\om')$ mentioned above after \eqref{eq:FstFqtFqs},
the mapping $[0,\infty)\ni t \mapsto F_t(\om,\om')\in[0,\infty)$
with $F_0(\om,\om'):=0$ is continuous, strictly
increasing and satisfies $\lim_{t\to\infty}F_t(\om,\om')=\infty$, proving ii).

Statement iii) is clear, so it remains to show iv). Let $\om \in \Omega$
satisfy the property in statement i). First, $F_0(\om,\cdot)=0$ is
$\mathcal{G}_0$-measurable, and for any $t > 0$, by i) we have
$\Lambda^\om \in \mathcal{G}_0\subset \mathcal{G}_t$, which together
with the $\mathcal{G}^0_t$-measurability of $F^n_s(\om,\cdot)$ for $n\in\bbN$
and $s\in[0,t]$ implies the $\mathcal{G}_t$-measurability of $F_t(\om,\cdot)$.
Next let $\om'\in\Lambda^\om$. \eqref{eq:Fn_additive} with $t=0$ results in
$F^n_s\bigl(\om,\theta_0(\om')\bigr)=F^n_s(\om,\om')$, $s\geq 0$, and then
by $(\om,\om')\in\Lambda$ we easily see $\theta_0(\om')\in\Lambda^\om$ and
$F_s(\om,\om')=F_0(\om,\om')+F_s\bigl(\om,\theta_0(\om')\bigr)$, $s\geq 0$.
For $t>0$, by \eqref{eq:Fn_additive}, $(\om,\om')\in\Lambda$,
\eqref{eq:FstFqtFqs} and \eqref{eq:FtFsFst} we have
\begin{align*}
 \liminf_{n\to \infty} F^n_s\bigl(\om,\theta_t(\om')\bigr)
	=\lim_{n\to \infty} \bigl(F^n_{s+t}(\om,\om') -F^n_t(\om,\om')\bigr)
	=F_{t,s+t}(\om,\om')\xrightarrow{s\downarrow 0}0
\end{align*}
and, for any $s\geq 0$ and $u\in[s,\infty)$,
\begin{align} \label{eq:F_additive}
 &F^n_u\bigl(\om,\theta_t(\om')\bigr) - F^n_s\bigl(\om,\theta_t(\om')\bigr)
	=F^n_{u+t}(\om,\om')-F^n_{s+t}(\om,\om') \nonumber \\
 &\mspace{198mu}\xrightarrow{n\to\infty}F_{s+t,u+t}(\om,\om')
	=F_{u+t}(\om,\om')-F_{s+t}(\om,\om'),
\end{align}
where the limit is a strictly increasing continuous function of $u\in[s,\infty)$
tending to $\infty$ as $u\to\infty$, proving in particular
$\bigl(\om,\theta_t(\om')\bigr)\in\Lambda$, i.e.\ $\theta_t(\om')\in\Lambda^\om$.
Finally, for $t,u>0$ and $s\in(0,u]$, \eqref{eq:F_additive} shows
$F_{s,u}\bigl(\om,\theta_t(\om')\bigr)=F_{u+t}(\om,\om')-F_{s+t}(\om,\om')$,
and letting $s\downarrow 0$ yields
$F_{u+t}(\om,\om')=F_{t}(\om,\om')+F_{u}\bigl(\om,\theta_t(\om')\bigr)$.
Therefore $\bigl(F_t(\om,\cdot)\bigr)_{t\geq 0}$ is a PCAF of $B$
in the strict sense with defining set $\Lambda^\om$.
\end{proof}

\section{The Revuz correspondence between $M$ and $F$} \label{app:revuz}

The purpose of this section is to give a proof of the following proposition,
which generalises Proposition~\ref{prop:revuz} to the LBM $\mathcal{B}^U$
killed upon exiting an open set $U\subset\bbR^2$.

\begin{prop} \label{prop:corrF_M}
$\bbP$-a.s., for any non-empty open set $U\subset \bbR^2$,
for all $x\in \bbR^2$ and all Borel measurable functions
$\eta: [0,\infty) \to [0,\infty]$ and $f: U \to [0,\infty]$,
\begin{align} \label{eq:FM}
 E_x\Bigl[ \int_0^{T_U} \eta(t) f(B_t) \, dF_t \Bigr]
	=\int_0^\infty \int_{U} \eta(t) f(y) q^U_t(x,y) \, M(dy)\, dt,
\end{align}
where $q^U_t(x,y)$ denotes the jointly continuous transition density of $B^U$
as in \eqref{eq:killed_green_U}.
\end{prop}

We need to prepare a few preliminary facts. First, by \cite[Theorem~2.2]{GRV13},
$\bbP$-a.s., for any $\varepsilon>0$ and any $R\geq 1$ there exists
$C_{21}=C_{21}(X,\gamma,R,\varepsilon)>0$ such that
\begin{align} \label{eq:voldec_n}
 M_n\bigl(B(x,r)\bigr) \leq C_{21} r^{\alpha_2-\varepsilon},
	\qquad \forall x\in B(R), \, r\in(0,1], \, n\in\bbN.
\end{align}

In the rest of this section, we fix any environment $\om \in \Omega$
such that $(M_n)_{n\geq 1}$ converges to $M$ vaguely on $\bbR^2$, the
conclusions of Proposition~\ref{prop:pcaf} i), iv) hold and \eqref{eq:voldec_n}
is valid for all $\varepsilon>0$ and $R\geq 1$.
Then by Proposition~\ref{prop:pcaf} i), ii), for all $x\in\bbR^2$,
\begin{align} \label{eq:vague_Fn_F}
 \text{$(dF^n_s)_{n\geq 1}$ converges to $dF_s$ weakly on $[t,u]$ for any $0<t\leq u$,}
	\quad\text{$P_x$-a.s.}
\end{align}

\begin{lemma} \label{lem:unif_int}
For any non-empty open set $U\subset \bbR^2$, any $x\in \bbR^2$,
any $t>0$ and any bounded Borel measurable function
$f: U \rightarrow [0,\infty)$ with $f^{-1}\bigl((0,\infty)\bigr)$ bounded,
$\bigl\{ \int_0^{T_{U}\wedge t} f(B_s) \, dF^n_s \bigr\}_{n\geq 1}$
is uniformly $P_x$-integrable.
\end{lemma}
\begin{proof}
It suffices to prove that
\begin{align} \label{eq:unif_L2}
 \sup_{n\geq 1}
	E_x\biggl[ \Bigl( \int_0^{T_{U}\wedge t} f(B_s)\, dF^n_s\Bigr)^2 \biggr]
	<\infty.
\end{align}
For any Borel measurable $h: U \rightarrow [0,\infty]$,
the Markov property of $B$ yields
\begin{align*}
 E_x\biggl[ \Bigl( \int_0^{T_{U}\wedge t} h(B_s)\, ds\Bigr)^2 \biggr]
	\leq 2\int_{U} \int_{U} h(y) h(z) \int_0^t \int_s^t q_s(x,y) q_{u-s}(y,z) \,  du \, ds\, dz\, dy.
\end{align*}
Then since
\begin{align*}
 \int_0^t \int_s^t q_s(x,y) q_{u-s}(y,z) \,  du \, ds  
	&\leq \int_0^t q_s(x,y) \, ds \int_0^t q_u(y,z) \, du \\
	&=\frac{1}{4\pi^2} \int_0^{t/|y-x|^2} s^{-1} e^{-\frac{1}{2s}} ds
		\int_0^{t/|z-y|^2} u^{-1} e^{-\frac{1}{2u}} du \\
& \leq \frac{1}{4\pi^2} \Bigl( 1+ \log^+\frac{t}{|y-x|^2} \Bigr)
	\Bigl( 1+ \log^+\frac{t}{|z-y|^2} \Bigr),
\end{align*}
where $\log^+=\log( \cdot \vee 1)$, setting
$h(y):=f(y)\exp\bigl(\gamma X_n(y)-\frac{\gamma^2}{2} \bbE[X_n(y)^2]\bigr)$,
recalling \eqref{eq:dfn_Fn} and \eqref{eq:dfn_Mn} and choosing
$R\geq 1$ such that $\{x\}\cup f^{-1}\bigl((0,\infty)\bigr)\subset B(R)$,
we obtain
\begin{align} 
 &E_x\biggl[ \Bigl( \int_0^{T_{U}\wedge t} f(B_s)\, dF^n_s\Bigr)^2 \biggr]
	=E_x\biggl[ \Bigl( \int_0^{T_{U}\wedge t} h(B_s)\, ds\Bigr)^2 \biggr] \nonumber \\
 &\mspace{30mu}\leq \frac{1}{2\pi^2} \int_{U}\int_{U} h(y) h(z)
	\Bigl( 1+ \log^+\frac{t}{|y-x|^2} \Bigr) \Bigl( 1+ \log^+\frac{t}{|z-y|^2} \Bigr)
	\, dz \, dy \nonumber \\
 &\mspace{30mu}\leq \frac{\|f\|_{\infty}^2}{2\pi^2} \int_{B(R)} \int_{B(R)}
	\Bigl( 1+ \log^+\frac{t}{|y-x|^2} \Bigr) \Bigl( 1+ \log^+\frac{t}{|z-y|^2} \Bigr)
	\, M_n(dz) \, M_n(dy).
	\label{eq:unif_L2_prf}
 \end{align}
Using \eqref{eq:voldec_n} with $\varepsilon=\alpha_{2}/2$,
for all $y\in B(R)$ and $n\geq 1$ we further get
\begin{align}
 &\int_{B(R)} \Bigl( 1+ \log^+\frac{t}{|z-y|^2} \Bigr) \, M_n(dz) \nonumber \\
 &\mspace{30mu}\leq M_n\bigl(B(R)\bigr) + \sum_{k=0}^\infty
	\int_{B(y,2^{1-k}R)\setminus B(y,2^{-k} R)} \log^+\frac{t}{(2^{-k}R)^2} \, M_n(dz)
	\nonumber \\
 &\mspace{30mu}\leq C + C \sum_{k=0}^\infty \Bigl(2k+\log^+ \frac{t}{R^{2}}\Bigr)(2^{1-k}R)^{\alpha_{2}/2}
	=:C'(X,\gamma,R,t)<\infty
	\label{eq:Mnlog}
\end{align}
for some constant $C=C(X,\gamma,R)>0$. \eqref{eq:Mnlog} is in fact valid
with $y=x$ by $x\in B(R)$, and then \eqref{eq:unif_L2} is immediate
from \eqref{eq:unif_L2_prf} and \eqref{eq:Mnlog}, completing the proof.
\end{proof}

Now we prove Proposition~\ref{prop:corrF_M} on the basis of
\eqref{eq:vague_Fn_F}, Lemma~\ref{lem:unif_int} and
the vague convergence on $\bbR^2$ of $M_n$ to $M$.

\begin{proof} [Proof of Proposition~\ref{prop:corrF_M}]
By a monotone class argument it suffices to consider continuous functions
$\eta$ and $f$ with compact supports in $(0,\infty)$ and $U$, respectively.
First note that by \eqref{eq:dfn_Fn}, Fubini's theorem and \eqref{eq:dfn_Mn}
we have for every $n\in \bbN$,
\begin{align}\label{eq:FM_n}
 E_x\Bigl[ \int_0^{T_U} \eta(t) f(B_t) \, dF^n_t \Bigr]
	=\int_0^\infty \int_{U} \eta(t) f(y) q^U_t(x,y) \, M_n(dy) \, dt,
\end{align}
and we need to show that letting $n\to\infty$ on both sides of
\eqref{eq:FM_n} results in \eqref{eq:FM}. The left-hand side of
\eqref{eq:FM_n} indeed converges to that of \eqref{eq:FM}
by \eqref{eq:vague_Fn_F} and the uniform $P_x$-integrability of
$\bigl\{ \int_0^{T_U}\eta(t) f(B_t) \, dF^n_t \bigr\}_{n\geq 1}$ implied
by Lemma~\ref{lem:unif_int}. On the other hand, the convergence of the
right-hand side of \eqref{eq:FM_n} to that of \eqref{eq:FM} follows from the
vague convergence on $\bbR^2$ of $M_n$ to $M$ together with the fact
that the function $U\ni y\mapsto \int_0^\infty \eta(t) f(y) q^U_t(x,y) \, dt$
is continuous with compact support in $U$ by virtue of dominated convergence
using the continuity of $q^U_t(x,\cdot)$ on $U$ and $0\leq q^U_t(x,y)\leq q_t(x,y)$.
Thus the proof of Proposition~\ref{prop:corrF_M} is complete.
\end{proof}

\section{Negative moments of the Liouville measure}

\begin{lemma} \label{lem:neg_mom}
Let $q>0$ and set
$\tilde \xi(q):=(2+\frac{\gamma^{2}}{2})q+\frac{\gamma^2}{2}q^2$. Then there
exists $c_4=c_4(\gamma,q)>0$ such that for any $x\in\bbR^2$ and any $r\in(0,1]$,
\begin{align} \label{eq:neg_mom}
 \bbE\Bigl[M\bigl(B(x,r)\bigr)^{-q}\Bigr] \vee
	\sup_{n\geq 1} \bbE\Bigl[M_n\bigl(B(x,r)\bigr)^{-q}\Bigr]
	\leq c_4 r^{-\tilde \xi(q)}.
\end{align}
\end{lemma}

\begin{proof}
Since the left-hand side of \eqref{eq:neg_mom} is independent of $x\in\bbR^2$
by the translation invariance of the laws of $M$ and $M_n$, $n\geq 1$,
it suffices to show \eqref{eq:neg_mom} for $x=0$.

The proof is based on a comparison with the moment estimates established in
\cite{RV10}, where the random Radon measure $M^0=M^0_{\gamma}$ on $\bbR^2$
associated with the covariance function $\gamma^2 g^{(m)}$ has been constructed
as follows. Note that $g^{(m)}$ can be written as $g^{(m)}(x,y)=h^{(m)}(x-y)$
with $h^{(m)}:=g^{(m)}(\cdot,0)$, which is easily seen from \eqref{eq:mGreenFunc}
to be of the form $h^{(m)}(x)= \log^+(|x|^{-1})+\Psi^{(m)}(x)$ for some
bounded continuous function $\Psi^{(m)}:\mathbb{R}^2 \to \bbR$. Define
$\psi:\bbR^2\to[0,\infty)$ by $\psi(x):=u*u(x)=\int_{\bbR^2}u(y)u(x-y)\,dy$
with $u(x):=\frac{3}{\pi}(1-|x|)^+$, so that $\psi$ is Lipschitz continuous,
$\psi|_{B(0,2)^{c}}=0$, $\int_{\bbR^2}\psi(x)\,dx=1$ and it is
positive definite, i.e.\ such that $\bigl(\psi(x-y)\bigr)_{x,y\in \Xi}$ is a
non-negative definite real symmetric matrix for any finite $\Xi\subset\bbR^2$.
Now for each $\varepsilon>0$, let $X^0_\varepsilon$ be a continuous Gaussian
field on $\bbR^2$ with mean $0$ and covariance
\begin{align*}
 \bbE\bigl[X_\varepsilon^0(x) X_\varepsilon^0(y) \bigr]
	= \psi_\varepsilon * h^{(m)}(x-y)
\end{align*}
for
$\psi_{\varepsilon}:=\varepsilon^{-2} \psi\bigl(\varepsilon^{-1}(\cdot)\bigr)$,
where such $X^0_\varepsilon$ can be constructed in exactly the same way as that described after
\eqref{eq:def_kn} since $\psi_\varepsilon * h^{(m)}$ is easily shown to be
positive definite and Lipschitz continuous. Then \cite[Theorem~2.1]{RV10}
(see also \cite[Theorem~3.2]{RV13a}) states that, as $\varepsilon\downarrow 0$,
the associated random Radon measure $M^0_{\varepsilon}=M^0_{\gamma,\varepsilon}$
on $\bbR^2$ defined by
\begin{align*}
 M^0_{\varepsilon}(dx)
	:=\exp\Bigl(\gamma X^0_{\varepsilon}(x)-\tfrac{\gamma^2}{2}
	\bbE\bigl[X^0_\varepsilon(x)^2\bigr]\Bigr) \, dx
\end{align*}
converges to some $M^0=M^0_{\gamma}$ in law in the space $\mathcal{M}(\bbR^2)$
of Radon measures on $\bbR^2$ equipped with the topology of vague convergence,
and $M^0$ satisfies the moment estimates as in \eqref{eq:neg_mom} by
\cite[Proposition~3.7]{RV10}.

Returning to \eqref{eq:def_kn}, for each $n\geq 1$ define
$h^{(m)}_{n}:=\sum_{k=1}^{n}g^{(m)}_{k}(\cdot,0)$, which is the covariance
kernel of $X_n=\sum_{k=1}^n Y_k$, and let $R\geq 1$ and $n\in\bbN$. Then
$h^{(m)}_{n+1}-h^{(m)}_{n}$ is $(0,\infty)$-valued and continuous,
$\lim_{\varepsilon\downarrow 0}\psi_{\varepsilon}*h^{(m)}_{n+1}=h^{(m)}_{n+1}$
uniformly on $\bbR^2$ by the uniform continuity of $h^{(m)}_{n+1}$ on $\bbR^2$,
and $h^{(m)}_{n+1}(x)<h^{(m)}(x)$ for any $x\in\bbR^2$,
so that there exists $\varepsilon_0>0$ such that
for all $\varepsilon\in(0,\varepsilon_0]$,
\begin{align} \label{eq:comparison_covariance}
 h^{(m)}_{n}(x)\leq \psi_\varepsilon * h^{(m)}_{n+1}(x)
	\leq \psi_\varepsilon * h^{(m)}(x), \qquad \forall x\in B(0,2R).
\end{align}
Let $f:\bbR^2\to[0,\infty)$ be continuous and satisfy $f|_{B(0,R)^{c}}=0$
and let $\eta:[0,\infty)\to \bbR$ be bounded, continuous and convex.
Also let $\varepsilon\in(0,\varepsilon_0]$. Then by
\eqref{eq:comparison_covariance}, we can apply Kahane's convexity inequality
(see \cite[Theorem~2.1]{RV13a} or \cite{Ka85}) to get
\begin{align*}
 \bbE\biggl[ \eta\biggl(\sum_{x\in k^{-1}\bbZ^2} \frac{f(x)}{k^{2}} e^{\gamma X_n(x)-\tfrac{\gamma^2}{2} \bbE[X_n(x)^2]}\biggr) \biggr]
	\leq \bbE\biggl[ \eta\biggl(\sum_{x\in k^{-1}\bbZ^2} \frac{f(x)}{k^{2}} e^{\gamma X^0_\varepsilon(x)-\tfrac{\gamma^2}{2} \bbE[X^0_\varepsilon(x)^2]}\biggr) \biggr]
\end{align*}
for all $k\in \bbN$, and by using dominated convergence to let $k\to\infty$, we obtain
\begin{align*}
 \bbE\Bigl[ \eta\Bigl( \int_{\bbR^2} f(x)e^{\gamma X_n(x)-\tfrac{\gamma^2}{2} \bbE[X_n(x)^2]} \, dx \Bigr) \Bigr]
	\leq \bbE\Bigl[ \eta\Bigl( \int_{\bbR^2} f(x)e^{\gamma X^0_\varepsilon(x)-\tfrac{\gamma^2}{2} \bbE[X^0_\varepsilon(x)^2]} \, dx \Bigr) \Bigr],
\end{align*}
which means that
$\bbE\bigl[ \eta\bigl( \Phi_{f}(M_n)\bigr) \bigr]
	\leq \bbE\bigl[ \eta\bigl( \Phi_{f}(M^0_\varepsilon)\bigr) \bigr]$
for the continuous function $\Phi_{f}:\mathcal{M}(\bbR^2)\to[0,\infty)$ given by
$\Phi_{f}(\mu):=\int_{\bbR^2}f\,d\mu$. Now since $M^0_\varepsilon$ converges
in law to $M^0$ as $\varepsilon\downarrow 0$ and
$\eta\circ\Phi_{f}:\mathcal{M}(\bbR^2)\to\bbR$ is bounded and continuous,
letting $\varepsilon\downarrow 0$ yields
\begin{align} \label{eq:E_eta_Phi_Mn}
 \bbE\bigl[ \eta\bigl( \Phi_{f}(M_n)\bigr) \bigr]
	\leq \lim_{\varepsilon\downarrow 0}\bbE\bigl[ \eta\bigl( \Phi_{f}(M^0_\varepsilon)\bigr) \bigr]
	= \bbE\bigl[ \eta\bigl( \Phi_{f}(M^0)\bigr) \bigr],
	\qquad \forall n\in\bbN,
\end{align}
whose limit as $n\to\infty$ results in
\begin{align} \label{eq:E_eta_Phi_M}
 \bbE\bigl[ \eta\bigl( \Phi_{f}(M)\bigr) \bigr]
	\leq \bbE\bigl[ \eta\bigl( \Phi_{f}(M^0)\bigr) \bigr]
\end{align}
by dominated convergence together with the fact that
$\lim_{n\to\infty}M_n = M$ in $\mathcal{M}(\bbR^2)$ $\bbP$-a.s.
Finally, letting $\eta(t)=\frac{1}{\Gamma(q)}\lambda^{q-1} e^{-\lambda t}$
with $\lambda>0$ and taking the $d\lambda$-integrals on $(0,\infty)$
in \eqref{eq:E_eta_Phi_Mn} and \eqref{eq:E_eta_Phi_M}, by
$\frac{1}{\Gamma(q)} \int_0^\infty \lambda^{q-1} e^{-\lambda t} \, d\lambda=t^{-q}$
we conclude that
\begin{align} \label{eq:E_Phi_M_minus_q}
 \bbE\bigl[ \Phi_{f}(M)^{-q}\bigr]\vee \sup_{n\geq 1} \bbE\bigl[ \Phi_{f}(M_n)^{-q}\bigr]
	\leq \bbE\bigl[ \Phi_{f}(M^0)^{-q}\bigr],
\end{align}
and \eqref{eq:neg_mom} for $x=0$ follows from \eqref{eq:E_Phi_M_minus_q}
with $f(y)=(2-2|y|/r)^{+}\wedge 1$ and the corresponding bound for
$\bbE\bigl[ M^0\bigl(B(0,r/2)\bigr)^{-q}\bigr]$ implied by \cite[Proposition~3.7]{RV10}. 
\end{proof}

\subsection*{Acknowledgements}
This paper was written while the second author was visiting the University of Bonn
in the summer term 2014. He thanks Kobe University for its financial and
administrative supports for his visit. He also would like to express his deepest
gratitude toward the stochastics research groups of the University of Bonn for
their heartfelt hospitality.
The topic of Liouville Brownian motion was suggested to the authors
by Karl-Theodor Sturm, for which they would like to thank him.

\bibliographystyle{plain}
\bibliography{literature}

\end{document}